\documentclass{article}

\usepackage{arxiv}

\usepackage[utf8]{inputenc}
\usepackage[T1]{fontenc}
\usepackage{microtype}
\usepackage{amsmath,amssymb,amsfonts,amsthm,mathtools,bm, mathrsfs, subcaption, float, hyperref}
\numberwithin{equation}{section}

\usepackage{booktabs}
\usepackage{enumitem}
\usepackage[authoryear,sort&compress]{natbib}

\usepackage[nameinlink,noabbrev]{cleveref}
\usepackage{url}
\usepackage{doi}

\hypersetup{
pdftitle={Non-Asymptotic Variational Learning for Monotone Nonlinear Multiscale Elliptic Equations: Scale-Robust Primal--Dual Bounds and Strong-Form Statistical Ill-Conditioning
},
pdfauthor={Ronald Katende},
pdfsubject={Non-asymptotic approximation, generalization, and optimization theory for uniformly monotone nonlinear multiscale elliptic equations
},
pdfkeywords={physics-informed neural networks, scientific machine learning, nonlinear elliptic equations, periodic homogenization, multiscale partial differential equations, variational neural methods, primal-dual learning, Rademacher complexity, non-asymptotic error bounds, statistical ill-conditioning } }

\newcommand{\R}{\mathbb{R}}
\newcommand{\E}{\mathbb{E}}
\newcommand{\Pp}{\mathbb{P}}
\newcommand{\Om}{\Omega}
\newcommand{\eps}{\varepsilon}
\newcommand{\V}{V}
\newcommand{\J}{\mathcal{J}}
\newcommand{\Jh}{\widehat{\mathcal{J}}}
\newcommand{\A}{\mathcal{A}}
\newcommand{\B}{\mathbb{B}}

\newcommand{\Radh}{\widehat{\mathfrak{R}}}
\newcommand{\norm}[1]{\left\lVert #1\right\rVert}
\newcommand{\abs}[1]{\left|#1\right|}

\newcommand{\dd}{\,\mathrm{d}}
\newcommand{\grad}{\nabla}
\newcommand{\diver}{\nabla\!\cdot}
\newcommand{\esssup}{\operatorname*{ess\,sup}}

\newtheorem{theorem}{Theorem}[section]
\newtheorem{lemma}[theorem]{Lemma}
\newtheorem{proposition}[theorem]{Proposition}

\theoremstyle{definition}
\newtheorem{assumption}[theorem]{Assumption}
\newtheorem{definition}[theorem]{Definition}

\theoremstyle{remark}
\newtheorem{remark}[theorem]{Remark}

\crefname{figure}{Figure}{Figures}
\Crefname{figure}{Figure}{Figures}

\crefname{subfigure}{Figure}{Figures}
\Crefname{subfigure}{Figure}{Figures}

\crefname{table}{Table}{Tables}
\Crefname{table}{Table}{Tables}

\crefname{section}{Section}{Sections}
\Crefname{section}{Section}{Sections}

\crefname{subsection}{Section}{Sections}
\Crefname{subsection}{Section}{Sections}

\crefname{equation}{Equation}{Equations}
\Crefname{equation}{Equation}{Equations}

\crefname{theorem}{Theorem}{Theorems}
\Crefname{theorem}{Theorem}{Theorems}

\crefname{lemma}{Lemma}{Lemmas}
\Crefname{lemma}{Lemma}{Lemmas}

\crefname{proposition}{Proposition}{Propositions}
\Crefname{proposition}{Proposition}{Propositions}

\crefname{corollary}{Corollary}{Corollaries}
\Crefname{corollary}{Corollary}{Corollaries}

\crefname{assumption}{Assumption}{Assumptions}
\Crefname{assumption}{Assumption}{Assumptions}

\crefname{definition}{Definition}{Definitions}
\Crefname{definition}{Definition}{Definitions}

\crefname{example}{Example}{Examples}
\Crefname{example}{Example}{Examples}

\crefname{remark}{Remark}{Remarks}
\Crefname{remark}{Remark}{Remarks}

\title{
Non-Asymptotic Variational Learning for Monotone Nonlinear Multiscale
Elliptic Equations:
Scale-Robust Primal--Dual Bounds and Strong-Form Statistical Ill-Conditioning
}

\author{
Ronald Katende \\
Department of Mathematics \\
Kabale University \\
Kikungiri Hill, Katuna Road, Kabale, Uganda \\
\texttt{rkatende@kab.ac.ug}
}

\date{}

\renewcommand{\headeright}{}
\renewcommand{\undertitle}{}
\renewcommand{\shorttitle}{
Non-Asymptotic Variational Learning for Monotone Nonlinear Multiscale Elliptic Equations
}

\begin{document}

\maketitle
\begin{abstract}
	We develop a non-asymptotic approximation, sampling, and finite-iteration optimization theory for variational physics-informed approximation of uniformly monotone nonlinear multiscale elliptic equations. For boundary-compatible neural feature classes, the population error splits into approximation, empirical quadrature, and projected-gradient terms, with all non-approximation constants uniform in the microscopic scale \(\varepsilon\).
	
	Assuming a quantitative corrected \(H^1\)-estimate, a two-scale state class yields
	\[
	\mathcal A_m^\varepsilon
	\le C\bigl(\varepsilon+\Phi_{0,m_0}^2+\Phi_{1,m_1}^2\bigr)
	\]
	in arbitrary dimension. We further introduce a convex primal--dual physics loss whose population value is a computable upper certificate for the state error. With additional flux-corrector regularity, a divergence-compatible two-scale flux class gives a certified state--flux bound combining \(O(\varepsilon)\) approximation, state and flux feature errors, empirical sampling error, and an \(O(K^{-1})\) optimization term.
	
	In contrast, for general periodic nonlinear fluxes satisfying a natural nondegeneracy condition, the empirical Rademacher complexities of strong-residual and squared-residual classes are bounded below by constant multiples of \((\varepsilon\sqrt N)^{-1}\) and \((\varepsilon^2\sqrt N)^{-1}\), respectively. These optimizer-independent lower bounds hold in every spatial dimension.
	
	Numerical experiments confirm the predicted \(\varepsilon\)- and \(N\)-scalings for nonlinear fluxes in \(d=1,2,3\), validate every computed primal--dual certificate, and show that corrector-enriched classes substantially reduce energy and \(H^1\) errors as the microscopic scale is refined.
\end{abstract}

\keywords{
variational physics-informed learning
\and nonlinear elliptic equations
\and periodic homogenization
\and primal--dual gap
\and Rademacher complexity
\and statistical ill-conditioning
\and corrector-enriched approximation
\and non-asymptotic error bounds
}

\section{Introduction}
\label{sec:introduction}

Physics-informed neural methods approximate PDE solutions by minimizing sampled physics losses. A non-asymptotic analysis must therefore control three distinct errors: approximation by the trial class, replacement of the population loss by finitely many samples, and incomplete numerical optimization. Stability-based generalization estimates for PINNs have been developed for several classes of PDEs \citep{MishraMolinaro2023}. Variational neural methods, including the Deep Ritz and VPINN frameworks, avoid the highest-order derivatives appearing in strong residuals \citep{EYu2018,KharazmiZhangKarniadakis2019}. These theories do not by themselves resolve how a microscopic coefficient changes the finite-sample loss class for a genuinely nonlinear divergence-form equation.

The distinction is formulation-dependent. If the flux is $A(x/\eps,\grad v)$, then the strong residual contains 
\[
\diver_x A\!\left(\frac{x}{\eps},\grad v\right)
=
\frac1\eps
\diver_yA\!\left(\frac{x}{\eps},\grad v\right)
+
D_\xi A\!\left(\frac{x}{\eps},\grad v\right):D^2v.
\]
Thus coefficient differentiation creates an explicit $\eps^{-1}$ term. For oscillatory linear elliptic equations, growth of order $\eps^{-2}$ has been identified in a neural tangent-kernel quantity \citep{GangalKimCarney2023}. That is a statement about local parameter-space training dynamics. It does not imply a lower bound for finite-sample fluctuations of the residual or squared-residual function classes. Conversely, a function-class lower bound does not imply an optimization rate. The present paper separates these effects.

The upper theory is variational. We first prove a scale-robust approximation--sampling--optimization decomposition for uniformly monotone nonlinear multiscale energies. We then close the remaining approximation term in arbitrary dimension by a corrector-enriched two-scale trial class. Periodic correctors and quantitative convergence estimates for nonlinear monotone operators are classical \citep{DalMasoDefranceschi1990,Bystrom2001,WangXuZhao2018}; the contribution here is their scale-explicit integration into a finite-sample, finite-iteration learning bound.

We next derive a first-order primal--dual physics loss. Primal--dual gaps are known to provide reliable energy certificates for uniformly convex minimization problems \citep{BartelsMilicevic2020}. Related mixed-residual and first-order least-squares neural formulations have been analysed for linear elliptic equations and general linear first-order systems \citep{LiTaiYangZhu2022,BersetcheBorthagaray2022, 	OpschoorPetersenSchwab2024}. Their objectives and stability mechanisms differ from the nonlinear Fenchel-gap formulation developed here. In particular, they do not provide the combination of a nonlinear periodic state--flux certificate, \(\eps\)-uniform finite-sample constants, and a divergence-compatible multiscale flux closure proved below. Here the gap is combined with finite neural state and flux classes, uniform empirical quadrature, and projected-gradient convergence. A nonlinear flux-corrector construction then closes both state and flux approximation terms while preserving $\eps$-uniform divergence envelopes.

Finally, we prove a general strong-form lower bound. For any fixed smooth trial state whose microscopic flux divergence is nondegenerate, a two-point residual class has empirical Rademacher complexity of order at least $\eps^{-1}N^{-1/2}$, while the corresponding squared-loss class has order at least $\eps^{-2}N^{-1/2}$. The result holds for nonlinear periodic fluxes in arbitrary dimension and is independent of the optimizer.

\paragraph{Main contributions.}
The paper proves four results.
\begin{enumerate}[leftmargin=2em]
\item A non-asymptotic primal variational estimate
\[
\norm{v_{\widehat\beta_K}-u^\eps}_{\V}^2
\leq
\frac2\alpha
\left[
\A_m^\eps
+
2\Delta_{N,\delta}
+
\frac{2L_{\mathrm{par}}R^2}{K}
\right],
\]
whose stability, sampling, and optimization constants are independent of
$\eps$.

\item Under quantitative periodic-corrector regularity, a multidimensional closure
\[
\A_m^\eps
\leq
C_{\mathrm{ms}}
\left(
\eps+\Phi_{0,m_0}^2+\Phi_{1,m_1}^2
\right).
\]

\item A convex primal--dual end-to-end theorem with a computable gap certificate and a divergence-compatible flux closure
\[
\mathfrak Q_n^\eps
\leq
C_{\mathrm{flux}}
\left(
\eps+\Psi_{0,n_0}^2+\Psi_{1,n_1}^2
\right).
\]

\item General-dimensional lower bounds
\[
\Radh_N(\mathcal R_{\eps,\star};S)
\gtrsim
\frac1{\eps\sqrt N},
\qquad
\Radh_N(\mathcal L_{\eps,\star};S)
\gtrsim
\frac1{\eps^2\sqrt N}
\]
for the strong residual and squared strong loss.
\end{enumerate}

\paragraph{Scope.}
The theory concerns uniformly monotone nonlinear divergence-form equations with convex energy structure. The optimized models are finite feature classes that remain linear in their trainable coefficients; this gives a global finite-iteration guarantee. The paper does not claim global optimization of unrestricted all-weights deep networks, nor does it cover non-monotone energies, fully nonlinear operators, or coupled systems. The strong-form lower bound is established for a class containing one \(\eps\)-independent trial state whose microscopic flux divergence is nondegenerate. It does not exclude coefficient-adapted or \(\eps\)-dependent trial families that suppress this microscopic divergence. Determining when such adapted architectures eliminate the obstruction, and what approximation or statistical complexity is required to do so, is a distinct open problem.

\section{Primal variational learning}
\label{sec:primal}

\subsection{Nonlinear multiscale energy}

Let $\Om\subset\R^d$ be bounded and Lipschitz, let $Y=(0,1)^d$, and set
\[
\V:=H_0^1(\Om),
\qquad
\norm{v}_{\V}:=\norm{\grad v}_{L^2(\Om)}.
\]
For $0<\eps\leq1$ and $f\in L^\infty(\Om)$, consider
\begin{equation}
\label{eq:pde-general}
-\diver\!\left[
\partial_\xi F\!\left(x,\frac{x}{\eps},\grad u^\eps\right)
\right]
+
\partial_sG(x,u^\eps)
=f
\quad\text{in }\Om,
\qquad
u^\eps=0
\quad\text{on }\partial\Om.
\end{equation}
Its energy is
\begin{equation}
\label{eq:energy-general}
\J_\eps(v)
:=
\int_\Om
\left[
F\!\left(x,\frac{x}{\eps},\grad v(x)\right)
+
G(x,v(x))
-
f(x)v(x)
\right]\dd x.
\end{equation}

\begin{assumption}[Uniform monotone structure]
\label{ass:structure}
The functions $F$ and $G$ are measurable in their spatial variables and
twice continuously differentiable in $\xi$ and $s$, respectively. There
exist $\alpha>0$ and $\mu\geq0$, independent of $\eps$, such that
\begin{equation}
\label{eq:hessian-lower}
\partial_{\xi\xi}^2F(x,y,\xi)\succeq\alpha I_d,
\qquad
\partial_{ss}^2G(x,s)\geq\mu
\end{equation}
for almost every $(x,y)$ and every $(\xi,s)$. For every $\eps\in(0,1]$,
the functional $\J_\eps:\V\to(-\infty,\infty]$ is proper, coercive, and
sequentially weakly lower semicontinuous.
\end{assumption}

Under \cref{ass:structure}, $\J_\eps$ has a unique minimizer
$u^\eps\in\V$, and its Euler equation is \eqref{eq:pde-general}. A
representative genuinely nonlinear example is
\[
F(y,\xi)
=
\frac12\xi^\top A(y)\xi
+
\gamma\left(\sqrt{1+\abs{\xi}^2}-1\right),
\qquad
G(s)=\frac\lambda2s^2+\frac14s^4,
\]
where $A$ is periodic and uniformly elliptic, $\gamma\geq0$, and
$\lambda>0$.

\begin{lemma}[Scale-independent energy certificate]
\label{lem:energy-certificate}
Under \cref{ass:structure}, every $v\in\V$ satisfies
\begin{equation}
\label{eq:energy-certificate}
\J_\eps(v)-\J_\eps(u^\eps)
\geq
\frac\alpha2\norm{v-u^\eps}_{\V}^2
+
\frac\mu2\norm{v-u^\eps}_{L^2(\Om)}^2.
\end{equation}
\end{lemma}

\begin{proof}
Strong convexity gives, pointwise,
\begin{align*}
F(x,y,\xi)
&\geq
F(x,y,\eta)
+
\partial_\xi F(x,y,\eta)\cdot(\xi-\eta)
+
\frac\alpha2\abs{\xi-\eta}^2,
\\
G(x,r)
&\geq
G(x,s)
+
\partial_sG(x,s)(r-s)
+
\frac\mu2\abs{r-s}^2.
\end{align*}
Set $(\xi,\eta,r,s)=(\grad v,\grad u^\eps,v,u^\eps)$ and integrate.
The first-order terms cancel with the forcing term because $u^\eps$ is
the minimizer of $\J_\eps$.
\end{proof}

\subsection{Boundary-compatible finite feature classes}

Let $z_m:\Om\to\R^m$ satisfy $z_m=0$ on $\partial\Om$, and write
\[
Z_m(x):=\grad z_m(x)\in\R^{d\times m}.
\]
For $R>0$, define
\begin{equation}
\label{eq:trial-class}
v_\beta(x):=\beta^\top z_m(x),
\qquad
\beta\in\B_R^m
:=
\left\{
\beta\in\R^m:
\norm{\beta}_2\leq R
\right\}.
\end{equation}
Only the output coefficients are trained. Set
\begin{equation}
\label{eq:feature-envelopes}
B_0
:=
\esssup_{x\in\Om}\norm{z_m(x)}_2,
\qquad
B_1
:=
\esssup_{x\in\Om}\norm{Z_m(x)}_{\mathrm{op}}.
\end{equation}
Thus
\[
\abs{v_\beta}\leq RB_0,
\qquad
\abs{\grad v_\beta}\leq RB_1.
\]

On these ranges define
\begin{align}
M_F
&:=
\esssup_{x,y,\abs{\xi}\leq RB_1}
\abs{\partial_\xi F(x,y,\xi)},
&
L_F
&:=
\esssup_{x,y,\abs{\xi}\leq RB_1}
\norm{\partial_{\xi\xi}^2F(x,y,\xi)}_{\mathrm{op}},
\label{eq:F-constants}\\
M_G
&:=
\esssup_{x,\abs{s}\leq RB_0}
\abs{\partial_sG(x,s)},
&
L_G
&:=
\esssup_{x,\abs{s}\leq RB_0}
\abs{\partial_{ss}^2G(x,s)}.
\label{eq:G-constants}
\end{align}
Assume these quantities and
\begin{equation}
\label{eq:zero-envelope}
C_0
:=
\esssup_{x,y}\abs{F(x,y,0)}
+
\esssup_x\abs{G(x,0)}
\end{equation}
are finite.

For
\begin{equation}
\label{eq:integrand-class}
h_\beta^\eps(x)
:=
F\!\left(x,\frac{x}{\eps},Z_m(x)\beta\right)
+
G\bigl(x,z_m(x)^\top\beta\bigr)
-
f(x)z_m(x)^\top\beta,
\end{equation}
set
\begin{align}
L_h
&:=
M_FB_1
+
\left(
M_G+\norm f_{L^\infty(\Om)}
\right)B_0,
\label{eq:Lh}\\
M_h
&:=
C_0+RL_h.
\label{eq:Mh}
\end{align}
Then
\begin{equation}
\label{eq:integrand-Lipschitz}
\abs{
h_\beta^\eps(x)-h_{\widetilde\beta}^\eps(x)
}
\leq
L_h\norm{\beta-\widetilde\beta}_2,
\qquad
\abs{h_\beta^\eps(x)}\leq M_h.
\end{equation}
Neither constant contains a derivative in the fast variable.

Let $X_1,\ldots,X_N$ be independent uniform samples on $\Om$, and define
\begin{equation}
\label{eq:empirical-energy}
\Jh_{\eps,N}(v_\beta)
:=
\frac{\abs{\Om}}{N}
\sum_{i=1}^N
h_\beta^\eps(X_i).
\end{equation}

\begin{lemma}[Uniform empirical quadrature]
\label{lem:uniform-quadrature}
For $N\geq1$ and $0<\delta<1$, with probability at least $1-\delta$,
\begin{equation}
\label{eq:delta-bound}
\sup_{\beta\in\B_R^m}
\abs{
\J_\eps(v_\beta)-\Jh_{\eps,N}(v_\beta)
}
\leq
\Delta_{N,\delta},
\end{equation}
where
\begin{equation}
\label{eq:Delta-explicit}
\Delta_{N,\delta}
:=
\abs{\Om}
\left[
\frac{2L_hR}{N}
+
M_h
\sqrt{
\frac{
2\left[
m\log(1+2N)+\log(2/\delta)
\right]
}{N}
}
\right].
\end{equation}
The bound is independent of $\eps$.
\end{lemma}

\begin{proof}
Let $\mathcal N$ be an $(R/N)$-net of $\B_R^m$. A volumetric estimate
gives
\[
\abs{\mathcal N}\leq(1+2N)^m.
\]
For a fixed $\widetilde\beta\in\mathcal N$, Hoeffding's inequality and
\eqref{eq:integrand-Lipschitz} imply
\[
\Pp\left(
\abs{(\E-P_N)h_{\widetilde\beta}^\eps}>t
\right)
\leq
2\exp\left(
-\frac{Nt^2}{2M_h^2}
\right).
\]
A union bound controls every net point. For arbitrary
$\beta\in\B_R^m$, choose
$\widetilde\beta\in\mathcal N$ with
$\norm{\beta-\widetilde\beta}_2\leq R/N$. The Lipschitz estimate adds
at most $2L_hR/N$ when passing from the net to the full ball.
Multiplication by $\abs{\Om}$ proves \eqref{eq:Delta-explicit}.
\end{proof}

\subsection{Finite-iteration convergence}

Define
\[
\widehat j_{\eps,N}(\beta)
:=
\Jh_{\eps,N}(v_\beta).
\]
The empirical objective is convex and
\begin{equation}
\label{eq:parameter-smoothness}
\norm{
\grad^2\widehat j_{\eps,N}(\beta)
}_{\mathrm{op}}
\leq
L_{\mathrm{par}}
:=
\abs{\Om}
\left(
L_FB_1^2+L_GB_0^2
\right).
\end{equation}
Projected gradient descent
\begin{equation}
\label{eq:pgd}
\widehat\beta_{k+1}
=
\Pi_{\B_R^m}
\left(
\widehat\beta_k
-
\frac1{L_{\mathrm{par}}}
\grad\widehat j_{\eps,N}(\widehat\beta_k)
\right)
\end{equation}
therefore satisfies
\begin{equation}
\label{eq:pgd-gap}
\widehat j_{\eps,N}(\widehat\beta_K)
-
\min_{\beta\in\B_R^m}
\widehat j_{\eps,N}(\beta)
\leq
\frac{2L_{\mathrm{par}}R^2}{K}.
\end{equation}

Define
\begin{equation}
\label{eq:approximation-error}
\A_m^\eps
:=
\inf_{\beta\in\B_R^m}
\left[
\J_\eps(v_\beta)-\J_\eps(u^\eps)
\right].
\end{equation}

\begin{theorem}[Scale-robust primal end-to-end bound]
\label{thm:end-to-end}
Under \cref{ass:structure} and
\eqref{eq:F-constants}--\eqref{eq:zero-envelope}, let
$\widehat\beta_K$ be generated by \eqref{eq:pgd}. With probability at
least $1-\delta$,
\begin{equation}
\label{eq:end-to-end}
\boxed{
\norm{v_{\widehat\beta_K}-u^\eps}_{\V}^2
\leq
\frac2\alpha
\left[
\A_m^\eps
+
2\Delta_{N,\delta}
+
\frac{2L_{\mathrm{par}}R^2}{K}
\right].
}
\end{equation}
Every term except $\A_m^\eps$ is independent of $\eps$.
\end{theorem}

\begin{proof}
For $\eta>0$, choose $\beta_\eta\in\B_R^m$ such that
\[
\J_\eps(v_{\beta_\eta})-\J_\eps(u^\eps)
\leq
\A_m^\eps+\eta.
\]
On the event in \cref{lem:uniform-quadrature}, insert the empirical
energy twice:
\begin{align*}
\J_\eps(v_{\widehat\beta_K})-\J_\eps(u^\eps)
\leq{}&
\Jh_{\eps,N}(v_{\widehat\beta_K})
-
\Jh_{\eps,N}(v_{\beta_\eta})
\\
&+
2\Delta_{N,\delta}
+
\A_m^\eps+\eta.
\end{align*}
Use \eqref{eq:pgd-gap}, let $\eta\downarrow0$, and apply
\cref{lem:energy-certificate}.
\end{proof}

\section{Multidimensional approximation closure}
\label{sec:state-closure}

We now close $\A_m^\eps$ for periodic nonlinear fluxes in arbitrary
dimension. In this section,
\[
F=F(y,\xi),
\qquad
G=G(s),
\qquad
A(y,\xi):=\partial_\xi F(y,\xi),
\qquad
g(s):=\partial_sG(s).
\]
The equation is
\begin{equation}
\label{eq:md-pde}
-\diver A\!\left(\frac{x}{\eps},\grad u^\eps\right)
+
g(u^\eps)
=f
\quad\text{in }\Om,
\qquad
u^\eps=0
\quad\text{on }\partial\Om.
\end{equation}

\begin{assumption}[Periodic corrector regime]
\label{ass:corrector-regime}
The domain $\Om$ is $C^{1,1}$. The flux is $Y$-periodic in $y$,
$A(y,0)=0$, and there exist $0<\alpha\leq\Lambda$ such that
\begin{equation}
\label{eq:md-structure}
\bigl(
A(y,\xi)-A(y,\zeta)
\bigr)\cdot(\xi-\zeta)
\geq
\alpha\abs{\xi-\zeta}^2,
\qquad
\abs{
A(y,\xi)-A(y,\zeta)
}
\leq
\Lambda\abs{\xi-\zeta}.
\end{equation}
The reaction $g$ is monotone and Lipschitz on the common range of the
solutions and trial functions.

For each $\xi\in\R^d$, let
$N(\cdot,\xi)\in H_{\mathrm{per}}^1(Y)$ be the mean-zero solution of
\begin{equation}
\label{eq:md-cell-problem}
-\diver_y
A\bigl(
y,\xi+\grad_yN(y,\xi)
\bigr)
=0
\quad\text{in }Y,
\qquad
\int_YN(y,\xi)\dd y=0,
\end{equation}
and define
\begin{equation}
\label{eq:md-homogenized-flux}
\widehat A(\xi)
:=
\int_Y
A\bigl(
y,\xi+\grad_yN(y,\xi)
\bigr)\dd y.
\end{equation}
Let $u^0\in H_0^1(\Om)$ solve
\begin{equation}
\label{eq:md-homogenized-pde}
-\diver\widehat A(\grad u^0)+g(u^0)=f
\quad\text{in }\Om,
\qquad
u^0=0
\quad\text{on }\partial\Om.
\end{equation}

There are cutoff functions $\eta_\eps\in W^{1,\infty}(\Om)$ satisfying
\begin{equation}
\label{eq:md-cutoff}
0\leq\eta_\eps\leq1,
\qquad
\eta_\eps=0
\ \text{if }\operatorname{dist}(x,\partial\Om)\leq\eps,
\qquad
\eta_\eps=1
\ \text{if }\operatorname{dist}(x,\partial\Om)\geq2\eps,
\qquad
\norm{\grad\eta_\eps}_{L^\infty}
\leq
C_\eta\eps^{-1}.
\end{equation}
Set
\begin{equation}
\label{eq:state-corrector}
\chi^0(x,y)
:=
N\bigl(y,\grad u^0(x)\bigr),
\qquad
w^\eps(x)
:=
u^0(x)
+
\eps\eta_\eps(x)
\chi^0\!\left(x,\frac{x}{\eps}\right).
\end{equation}
Assume
\begin{equation}
\label{eq:corrector-regularity}
u^0\in W^{2,\infty}(\Om),
\qquad
\chi^0\in W^{1,\infty}(\Om\times Y),
\end{equation}
and the quantitative corrector estimate
\begin{equation}
\label{eq:md-quantitative-corrector}
\norm{u^\eps-w^\eps}_{H^1(\Om)}
\leq
C_{\mathrm{hom}}\sqrt\eps
\end{equation}
holds with $C_{\mathrm{hom}}$ independent of $\eps$.
\end{assumption}

\begin{remark}
Existence of the nonlinear corrector and qualitative convergence follow
from the periodic homogenization theory for monotone operators
\citep{DalMasoDefranceschi1990,Bystrom2001}. Quantitative
$O(\eps^{1/2})$ estimates on $C^{1,1}$ domains are available for
quasilinear periodic operators under corresponding structural hypotheses
\citep{WangXuZhao2018}. We state \eqref{eq:md-quantitative-corrector}
explicitly because it is the exact homogenization input used below.
\end{remark}

The estimate \eqref{eq:md-quantitative-corrector} is a stated quantitative homogenization hypothesis for the reaction--diffusion problem considered here. The cited quasilinear homogenization results supply the corresponding diffusion corrector estimates; the present learning theorem uses only \eqref{eq:md-quantitative-corrector} and does not require a broader homogenization claim.

Let
\[
\mathcal U_{m_0}
:=
\operatorname{span}
\{z_1,\ldots,z_{m_0}\}
\subset
H_0^1(\Om)\cap W^{1,\infty}(\Om)
\]
and
\[
\mathcal C_{m_1}
:=
\operatorname{span}
\{\Phi_1,\ldots,\Phi_{m_1}\}
\subset
W^{1,\infty}(\Om\times Y),
\]
where every $\Phi_\ell$ is $Y$-periodic. Define
\begin{equation}
\label{eq:md-approximation-errors}
\Phi_{0,m_0}
:=
\inf_{U\in\mathcal U_{m_0}}
\norm{U-u^0}_{H^1(\Om)},
\qquad
\Phi_{1,m_1}
:=
\inf_{C\in\mathcal C_{m_1}}
\norm{C-\chi^0}_{W^{1,\infty}(\Om\times Y)}.
\end{equation}
For $m=m_0+m_1$, set
\begin{equation}
\label{eq:md-trial-class}
\mathcal V_m^\eps
:=
\left\{
U(x)
+
\eps\eta_\eps(x)
C\!\left(x,\frac{x}{\eps}\right)
:
U\in\mathcal U_{m_0},
\ C\in\mathcal C_{m_1}
\right\}.
\end{equation}
The class is linear in its trainable coefficients. Moreover,
\begin{equation}
\label{eq:md-feature-gradient}
\grad\!\left[
\eps\eta_\eps(x)
\Phi_\ell\!\left(x,\frac{x}{\eps}\right)
\right]
=
\eps\grad\eta_\eps\,\Phi_\ell
+
\eps\eta_\eps\grad_x\Phi_\ell
+
\eta_\eps\grad_y\Phi_\ell,
\end{equation}
so its feature envelopes are independent of $\eps$.

\begin{theorem}[Multidimensional state approximation closure]
\label{thm:md-approximation-closure}
Under \cref{ass:structure,ass:corrector-regime}, assume that the
coefficient ball contains approximants realizing
\eqref{eq:md-approximation-errors}. Then there exists
$C_{\mathrm{ms}}>0$, independent of $\eps$, such that
\begin{equation}
\label{eq:md-approximation-bound}
\A_m^\eps
\leq
C_{\mathrm{ms}}
\left(
\eps
+
\Phi_{0,m_0}^2
+
\Phi_{1,m_1}^2
\right).
\end{equation}
Consequently, with probability at least $1-\delta$,
\begin{equation}
\label{eq:md-end-to-end}
\boxed{
\norm{v_{\widehat\beta_K}-u^\eps}_{\V}^2
\leq
\frac2\alpha
\left[
C_{\mathrm{ms}}
\left(
\eps+\Phi_{0,m_0}^2+\Phi_{1,m_1}^2
\right)
+
2\Delta_{N,\delta}
+
\frac{2L_{\mathrm{par}}R^2}{K}
\right].
}
\end{equation}
\end{theorem}

\begin{proof}
Choose $U_m\in\mathcal U_{m_0}$ and
$C_m\in\mathcal C_{m_1}$ arbitrarily close to the infima in
\eqref{eq:md-approximation-errors}, and define
\[
v_m^\eps(x)
:=
U_m(x)
+
\eps\eta_\eps(x)
C_m\!\left(x,\frac{x}{\eps}\right).
\]
Let $D_m:=C_m-\chi^0$. From \eqref{eq:md-cutoff} and
\eqref{eq:md-feature-gradient},
\[
\norm{
\eps\eta_\eps
D_m(\cdot,\cdot/\eps)
}_{H^1(\Om)}
\leq
C_{\mathrm{lift}}
\norm{D_m}_{W^{1,\infty}(\Om\times Y)},
\]
where $C_{\mathrm{lift}}$ is independent of $\eps$. Therefore
\[
\norm{v_m^\eps-w^\eps}_{H^1(\Om)}
\leq
\Phi_{0,m_0}
+
C_{\mathrm{lift}}\Phi_{1,m_1}.
\]
Together with \eqref{eq:md-quantitative-corrector},
\[
\norm{v_m^\eps-u^\eps}_{H^1(\Om)}
\leq
C_{\mathrm{hom}}\sqrt\eps
+
\Phi_{0,m_0}
+
C_{\mathrm{lift}}\Phi_{1,m_1}.
\]
Taylor's formula around the minimizer $u^\eps$, using the upper
curvature bounds on the trial range, gives
\[
\J_\eps(v_m^\eps)-\J_\eps(u^\eps)
\leq
\frac{L_F+L_G}{2}
\norm{v_m^\eps-u^\eps}_{H^1(\Om)}^2.
\]
The inequality $(a+b+c)^2\leq3(a^2+b^2+c^2)$ proves
\eqref{eq:md-approximation-bound}. The feature envelopes are
$\eps$-uniform by \eqref{eq:md-feature-gradient}, so
\cref{thm:end-to-end} yields \eqref{eq:md-end-to-end}.
\end{proof}

\section{Scale-robust primal--dual learning}
\label{sec:primal-dual}

The primal energy controls the state but does not directly certify a learned
flux. We now derive a first-order primal--dual formulation that remains
convex, provides a computable state-error certificate, and admits
$\eps$-uniform finite-sample and finite-iteration bounds.

\subsection{Convex duality and the gap certificate}

For almost every $(x,y)$, define
\begin{align}
F^*(x,y,p)
&:=
\sup_{\xi\in\R^d}
\left\{
p\cdot\xi-F(x,y,\xi)
\right\},
\label{eq:pd-F-star}\\
G^*(x,r)
&:=
\sup_{s\in\R}
\left\{
rs-G(x,s)
\right\}.
\label{eq:pd-G-star}
\end{align}

\begin{assumption}[Primal--dual curvature]
\label{ass:pd-regularity}
In addition to \cref{ass:structure}, assume $\mu>0$ and
\begin{equation}
\label{eq:pd-curvature}
\alpha I_d
\preceq
\partial_{\xi\xi}^2F(x,y,\xi)
\preceq
\overline L_FI_d,
\qquad
\mu
\leq
\partial_{ss}^2G(x,s)
\leq
\overline L_G
\end{equation}
for finite $\overline L_F,\overline L_G$ independent of $\eps$. Assume
also
\begin{equation}
\label{eq:pd-baseline-regularity}
\esssup_{x,y}
\abs{\partial_\xi F(x,y,0)}
<\infty,
\qquad
\partial_sG(\cdot,0)\in L^2(\Om).
\end{equation}
\end{assumption}

Then
\begin{equation}
\label{eq:pd-conjugate-curvature}
\frac1{\overline L_F}I_d
\preceq
\partial_{pp}^2F^*(x,y,p)
\preceq
\frac1\alpha I_d,
\qquad
\frac1{\overline L_G}
\leq
\partial_{rr}^2G^*(x,r)
\leq
\frac1\mu.
\end{equation}

Let
\[
H(\diver;\Om)
:=
\left\{
p\in L^2(\Om;\R^d):
\diver p\in L^2(\Om)
\right\}.
\]
For $p\in H(\diver;\Om)$, set
\begin{equation}
\label{eq:pd-rp}
r_p:=f+\diver p.
\end{equation}

\begin{definition}[Primal--dual gap]
\label{def:pd-gap}
For $v\in\V$ and $p\in H(\diver;\Om)$, define
\begin{equation}
\label{eq:pd-gap}
\mathfrak G_\eps(v,p)
:=
\J_\eps(v)
+
\int_\Om
\left[
F^*\!\left(x,\frac{x}{\eps},p\right)
+
G^*(x,r_p)
\right]\dd x.
\end{equation}
\end{definition}

\begin{theorem}[Scale-independent primal--dual certificate]
\label{thm:pd-certificate}
Let $u^\eps$ solve \eqref{eq:pde-general}. Then, for every
$v\in\V$ and $p\in H(\diver;\Om)$,
\begin{equation}
\label{eq:pd-reliability}
\mathfrak G_\eps(v,p)
\geq
\J_\eps(v)-\J_\eps(u^\eps)
\geq
\frac\alpha2\norm{v-u^\eps}_{\V}^2
+
\frac\mu2\norm{v-u^\eps}_{L^2(\Om)}^2.
\end{equation}
With
\begin{equation}
\label{eq:pd-exact-fields}
p^\eps
:=
\partial_\xi F\!\left(
x,\frac{x}{\eps},\grad u^\eps
\right),
\qquad
r^\eps
:=
\partial_sG(x,u^\eps),
\end{equation}
one has
\begin{equation}
\label{eq:pd-zero-gap}
p^\eps\in H(\diver;\Om),
\qquad
r^\eps=f+\diver p^\eps,
\qquad
\mathfrak G_\eps(u^\eps,p^\eps)=0.
\end{equation}
\end{theorem}

\begin{proof}
Since $v$ has zero trace and $f=r_p-\diver p$,
\[
\int_\Om fv\dd x
=
\int_\Om
\left(
p\cdot\grad v+r_pv
\right)\dd x.
\]
Hence
\begin{align}
\mathfrak G_\eps(v,p)
={}&
\int_\Om
\left[
F\!\left(x,\frac{x}{\eps},\grad v\right)
+
F^*\!\left(x,\frac{x}{\eps},p\right)
-
p\cdot\grad v
\right]\dd x
\notag\\
&+
\int_\Om
\left[
G(x,v)+G^*(x,r_p)-r_pv
\right]\dd x.
\label{eq:pd-fenchel-decomposition}
\end{align}
Both integrands are nonnegative by Fenchel--Young. Define
\[
\mathcal D_\eps(p)
:=
-
\int_\Om
\left[
F^*\!\left(x,\frac{x}{\eps},p\right)
+
G^*(x,r_p)
\right]\dd x.
\]
Equation \eqref{eq:pd-fenchel-decomposition} implies
$\mathcal D_\eps(p)\leq\J_\eps(v)$ for every $v$, hence
$\mathcal D_\eps(p)\leq\J_\eps(u^\eps)$. Thus
\[
\mathfrak G_\eps(v,p)
=
\J_\eps(v)-\mathcal D_\eps(p)
\geq
\J_\eps(v)-\J_\eps(u^\eps),
\]
and \cref{lem:energy-certificate} gives
\eqref{eq:pd-reliability}.

The PDE yields $-\diver p^\eps+r^\eps=f$. Fenchel equality holds at the
exact conjugate pairs:
\[
F(\grad u^\eps)+F^*(p^\eps)
=
p^\eps\cdot\grad u^\eps,
\qquad
G(u^\eps)+G^*(r^\eps)
=
r^\eps u^\eps.
\]
Substitution into \eqref{eq:pd-fenchel-decomposition} gives zero gap.
\end{proof}

\subsection{Finite state and flux classes}

Retain
\[
v_\beta=\beta^\top z_m,
\qquad
\beta\in\B_R^m.
\]
Let
$w_1,\ldots,w_n\in H(\diver;\Om)\cap L^\infty(\Om;\R^d)$ with
$\diver w_j\in L^\infty(\Om)$, and set
\[
W_n(x)
:=
\begin{bmatrix}
w_1(x)&\cdots&w_n(x)
\end{bmatrix},
\qquad
D_n(x)
:=
\begin{bmatrix}
\diver w_1(x)\\
\vdots\\
\diver w_n(x)
\end{bmatrix}.
\]
For $\gamma\in\B_S^n$, define
\begin{equation}
\label{eq:pd-dual-class}
p_\gamma(x):=W_n(x)\gamma,
\qquad
r_\gamma(x)
=
f(x)+D_n(x)^\top\gamma.
\end{equation}
Set
\begin{equation}
\label{eq:pd-dual-envelopes}
B_2
:=
\esssup_x\norm{W_n(x)}_{\mathrm{op}},
\qquad
B_3
:=
\esssup_x\norm{D_n(x)}_2.
\end{equation}

On
\[
\abs p\leq SB_2,
\qquad
\abs r\leq\norm{f}_{L^\infty(\Om)}+SB_3,
\]
define finite derivative envelopes $M_{F^*}$ and $M_{G^*}$, and set
\begin{align}
L_h^{\mathrm{pd}}
&:=
M_FB_1
+
(M_G+\norm{f}_{L^\infty(\Om)})B_0
+
M_{F^*}B_2
+
M_{G^*}B_3,
\label{eq:pd-Lh}\\
R_{\mathrm{pd}}
&:=
\sqrt{R^2+S^2}.
\end{align}
Assume also that
\begin{equation}
\label{eq:pd-zero-envelope}
C_0^{\mathrm{pd}}
:=
\esssup_{x,y}
\left|
F(x,y,0)+G(x,0)+F^*(x,y,0)+G^*(x,f(x))
\right|
<\infty,
\end{equation}
and define
\[
M_h^{\mathrm{pd}}
:=
C_0^{\mathrm{pd}}
+
R_{\mathrm{pd}}L_h^{\mathrm{pd}}.
\]

For $(\beta,\gamma)\in\B_R^m\times\B_S^n$, define
\begin{align}
h_{\beta,\gamma}^{\eps,\mathrm{pd}}(x)
:={}&
F\!\left(
x,\frac{x}{\eps},Z_m(x)\beta
\right)
+
G\bigl(x,z_m(x)^\top\beta\bigr)
-
f(x)z_m(x)^\top\beta
\notag\\
&+
F^*\!\left(
x,\frac{x}{\eps},W_n(x)\gamma
\right)
+
G^*\bigl(
x,f(x)+D_n(x)^\top\gamma
\bigr).
\label{eq:pd-integrand}
\end{align}
Then
\[
\abs{
h_{\beta,\gamma}^{\eps,\mathrm{pd}}
-
h_{\widetilde\beta,\widetilde\gamma}^{\eps,\mathrm{pd}}
}
\leq
L_h^{\mathrm{pd}}
\left\|
\begin{pmatrix}
\beta-\widetilde\beta\\
\gamma-\widetilde\gamma
\end{pmatrix}
\right\|_2,
\qquad
\abs{h_{\beta,\gamma}^{\eps,\mathrm{pd}}}
\leq
M_h^{\mathrm{pd}}.
\]

Define
\begin{equation}
\label{eq:pd-empirical-gap}
\widehat{\mathfrak G}_{\eps,N}(\beta,\gamma)
:=
\frac{\abs{\Om}}{N}
\sum_{i=1}^N
h_{\beta,\gamma}^{\eps,\mathrm{pd}}(X_i).
\end{equation}

\begin{lemma}[Uniform primal--dual empirical deviation]
\label{lem:pd-uniform-deviation}
With probability at least $1-\delta$,
\begin{equation}
\label{eq:pd-uniform-deviation}
\sup_{\beta\in\B_R^m,\gamma\in\B_S^n}
\left|
\mathfrak G_\eps(v_\beta,p_\gamma)
-
\widehat{\mathfrak G}_{\eps,N}(\beta,\gamma)
\right|
\leq
\Delta_{N,\delta}^{\mathrm{pd}},
\end{equation}
where
\begin{equation}
\label{eq:pd-Delta}
\Delta_{N,\delta}^{\mathrm{pd}}
:=
\abs{\Om}
\left[
\frac{
2L_h^{\mathrm{pd}}R_{\mathrm{pd}}
}{N}
+
M_h^{\mathrm{pd}}
\sqrt{
\frac{
2\left[
(m+n)\log(1+2N)+\log(2/\delta)
\right]
}{N}
}
\right].
\end{equation}
\end{lemma}

\begin{proof}
Apply the proof of \cref{lem:uniform-quadrature} to the product parameter
set, viewed as a subset of the Euclidean ball of radius
$R_{\mathrm{pd}}$ in $\R^{m+n}$.
\end{proof}

The empirical gap is convex in $(\beta,\gamma)$. Its Hessian is block
diagonal and satisfies
\begin{equation}
\label{eq:pd-Lpar}
L_{\mathrm{par}}^{\mathrm{pd}}
:=
\abs{\Om}
\max\left\{
\overline L_FB_1^2+\overline L_GB_0^2,
\frac{B_2^2}{\alpha}+\frac{B_3^2}{\mu}
\right\}.
\end{equation}
Projected gradient descent on
$\Theta_{R,S}:=\B_R^m\times\B_S^n$ therefore gives
\begin{equation}
\label{eq:pd-pgd-rate}
\widehat{\mathfrak G}_{\eps,N}(\beta_K,\gamma_K)
-
\min_{\Theta_{R,S}}
\widehat{\mathfrak G}_{\eps,N}
\leq
\frac{
2L_{\mathrm{par}}^{\mathrm{pd}}(R^2+S^2)
}{K}.
\end{equation}

\subsection{Separated primal and flux approximation}

Define
\begin{equation}
\label{eq:pd-approximation-error}
\mathfrak A_{m,n}^{\eps,\mathrm{pd}}
:=
\inf_{\beta\in\B_R^m,\gamma\in\B_S^n}
\mathfrak G_\eps(v_\beta,p_\gamma).
\end{equation}
Set
\begin{align}
\mathfrak E_m^\eps
&:=
\inf_{\beta\in\B_R^m}
\left[
\frac{\overline L_F}{2}
\norm{\grad(v_\beta-u^\eps)}_{L^2}^2
+
\frac{\overline L_G}{2}
\norm{v_\beta-u^\eps}_{L^2}^2
\right],
\label{eq:pd-primal-approx}\\
\mathfrak Q_n^\eps
&:=
\inf_{\gamma\in\B_S^n}
\left[
\frac1{2\alpha}
\norm{p_\gamma-p^\eps}_{L^2}^2
+
\frac1{2\mu}
\norm{\diver(p_\gamma-p^\eps)}_{L^2}^2
\right].
\label{eq:pd-dual-approx}
\end{align}

\begin{proposition}[Separated approximation bound]
\label{prop:pd-separated-approximation}
Under \cref{ass:pd-regularity},
\begin{equation}
\label{eq:pd-separated-bound}
\mathfrak A_{m,n}^{\eps,\mathrm{pd}}
\leq
\mathfrak E_m^\eps+\mathfrak Q_n^\eps.
\end{equation}
\end{proposition}

\begin{proof}
Define
\[
\Phi_\eps(p)
:=
\int_\Om
\left[
F^*\!\left(x,\frac{x}{\eps},p\right)
+
G^*(x,f+\diver p)
\right]\dd x.
\]
By zero gap,
\[
\mathfrak G_\eps(v,p)
=
\left[
\J_\eps(v)-\J_\eps(u^\eps)
\right]
+
\left[
\Phi_\eps(p)-\Phi_\eps(p^\eps)
\right].
\]
The first variations of $\J_\eps$ at $u^\eps$ and of $\Phi_\eps$ at
$p^\eps$ vanish. Taylor's formula and
\eqref{eq:pd-curvature}--\eqref{eq:pd-conjugate-curvature} give
\begin{align*}
\J_\eps(v)-\J_\eps(u^\eps)
&\leq
\frac{\overline L_F}{2}
\norm{\grad(v-u^\eps)}_{L^2}^2
+
\frac{\overline L_G}{2}
\norm{v-u^\eps}_{L^2}^2,
\\
\Phi_\eps(p)-\Phi_\eps(p^\eps)
&\leq
\frac1{2\alpha}
\norm{p-p^\eps}_{L^2}^2
+
\frac1{2\mu}
\norm{\diver(p-p^\eps)}_{L^2}^2.
\end{align*}
Minimize independently over the two parameter balls.
\end{proof}

\subsection{Divergence-compatible nonlinear flux closure}

We now work under \cref{ass:corrector-regime}. Define the microscopic
corrected flux
\begin{equation}
\label{eq:micro-flux}
P^0(x,y)
:=
A\left(
y,
\grad u^0(x)
+
\grad_yN\bigl(y,\grad u^0(x)\bigr)
\right).
\end{equation}
Its cell divergence vanishes and its cell average is
$\widehat A(\grad u^0(x))$.

For $\xi\in\R^d$, let
\begin{equation}
\label{eq:flux-discrepancy}
b_i(y,\xi)
:=
A_i\left(
y,\xi+\grad_yN(y,\xi)
\right)
-
\widehat A_i(\xi).
\end{equation}
Then
\[
\int_Yb(\cdot,\xi)\dd y=0,
\qquad
\diver_yb(\cdot,\xi)=0.
\]
A periodic flux corrector is a skew-symmetric tensor
$E_{ji}(y,\xi)=-E_{ij}(y,\xi)$ satisfying
\begin{equation}
\label{eq:flux-corrector}
b_i(y,\xi)
=
\partial_{y_j}E_{ji}(y,\xi).
\end{equation}
The existence of a skew-symmetric periodic flux corrector in
\(H_{\mathrm{per}}^1(Y)\), together with quantitative \(L^2\)-bounds for
its derivatives with respect to the macroscopic gradient variable, follows
from \citet[Lemma~2.5]{WangXuZhao2018}. The stronger
\(W^{1,\infty}\)-regularity required below is imposed separately and is not
claimed to follow from that lemma.

\begin{assumption}[Additional flux-corrector regularity]
\label{ass:flux-corrector}
The tensor $E$ can be chosen so that the field
\begin{equation}
\label{eq:flux-Q0}
Q_i^0(x,y)
:=
\partial_{\xi_k}E_{ji}
\bigl(
y,\grad u^0(x)
\bigr)
\,
\partial_{x_jx_k}u^0(x)
\end{equation}
belongs to $W^{1,\infty}(\Om\times Y;\R^d)$. The fields
$P^0$ and $Q^0$ are $Y$-periodic and bounded in
$W^{1,\infty}(\Om\times Y)$.
\end{assumption}

Define
\begin{equation}
\label{eq:corrected-flux}
\pi^\eps(x)
:=
P^0\!\left(x,\frac{x}{\eps}\right)
+
\eps
Q^0\!\left(x,\frac{x}{\eps}\right).
\end{equation}

\begin{lemma}[Exact divergence identity]
\label{lem:flux-divergence-identity}
Under \cref{ass:flux-corrector},
\begin{equation}
\label{eq:flux-divergence-identity}
\diver\pi^\eps
=
\diver\widehat A(\grad u^0)
=
g(u^0)-f
\quad\text{in }\Om.
\end{equation}
\end{lemma}

\begin{proof}
Write
\[
E_{ji}^\eps(x)
:=
E_{ji}\left(
\frac{x}{\eps},
\grad u^0(x)
\right).
\]
By \eqref{eq:flux-corrector} and the chain rule,
\[
\eps\partial_{x_j}E_{ji}^\eps
=
b_i\left(
\frac{x}{\eps},
\grad u^0
\right)
+
\eps Q_i^0\left(
x,\frac{x}{\eps}
\right).
\]
Hence
\[
\pi_i^\eps
=
\widehat A_i(\grad u^0)
+
\eps\partial_{x_j}E_{ji}^\eps.
\]
Taking $\partial_{x_i}$ and summing gives
\[
\diver\pi^\eps
=
\diver\widehat A(\grad u^0)
+
\eps
\partial_{x_i}\partial_{x_j}E_{ji}^\eps.
\]
The last term vanishes because the second derivatives are symmetric in
$(i,j)$ while $E_{ji}^\eps=-E_{ij}^\eps$. The homogenized equation
\eqref{eq:md-homogenized-pde} gives the final identity.
\end{proof}

\begin{proposition}[Flux corrector in $H(\diver)$]
\label{prop:flux-Hdiv-corrector}
Under
\cref{ass:corrector-regime,ass:flux-corrector}, there exists
$C_{\mathrm{fc}}>0$, independent of $\eps$, such that
\begin{equation}
\label{eq:flux-Hdiv-corrector}
\norm{p^\eps-\pi^\eps}_{L^2(\Om)}
+
\norm{\diver(p^\eps-\pi^\eps)}_{L^2(\Om)}
\leq
C_{\mathrm{fc}}\sqrt\eps.
\end{equation}
\end{proposition}

\begin{proof}
Let
\[
\Xi^0(x,y)
:=
\grad u^0(x)
+
\grad_yN\bigl(y,\grad u^0(x)\bigr).
\]
The cutoff structure in \eqref{eq:state-corrector} gives
\begin{equation}
\label{eq:gradient-corrector-to-Xi}
\norm{
\grad w^\eps
-
\Xi^0(\cdot,\cdot/\eps)
}_{L^2(\Om)}
\leq
C\sqrt\eps.
\end{equation}
Indeed, away from the boundary strip the difference is
$\eps\grad_x\chi^0$, while the remaining terms are supported in a strip
of measure $O(\eps)$ and are uniformly bounded by
\eqref{eq:md-cutoff} and \eqref{eq:corrector-regularity}.

By Lipschitz continuity of $A$,
\begin{align*}
\norm{
p^\eps
-
P^0(\cdot,\cdot/\eps)
}_{L^2}
&\leq
\Lambda
\norm{
\grad u^\eps-\grad w^\eps
}_{L^2}
\\
&\quad+
\Lambda
\norm{
\grad w^\eps-\Xi^0(\cdot,\cdot/\eps)
}_{L^2}
\\
&\leq
C\sqrt\eps.
\end{align*}
Since $\eps Q^0(\cdot,\cdot/\eps)$ is $O(\eps)$ in $L^2$,
\[
\norm{p^\eps-\pi^\eps}_{L^2}
\leq
C\sqrt\eps.
\]

For the divergence, use the exact equations:
\[
\diver p^\eps=g(u^\eps)-f,
\qquad
\diver\pi^\eps=g(u^0)-f.
\]
Hence
\[
\norm{\diver(p^\eps-\pi^\eps)}_{L^2}
\leq
L_G
\norm{u^\eps-u^0}_{L^2}.
\]
Now
\[
\norm{u^\eps-u^0}_{L^2}
\leq
\norm{u^\eps-w^\eps}_{L^2}
+
\eps
\norm{
\eta_\eps\chi^0(\cdot,\cdot/\eps)
}_{L^2}
\leq
C\sqrt\eps.
\]
This proves \eqref{eq:flux-Hdiv-corrector}.
\end{proof}

Let
\[
\mathcal P_{n_0}
\subset
W^{1,\infty}_{\mathrm{per}}
(\Om\times Y;\R^d),
\qquad
\mathcal Q_{n_1}
\subset
W^{1,\infty}_{\mathrm{per}}
(\Om\times Y;\R^d)
\]
be finite-dimensional linear spaces such that
\begin{equation}
\label{eq:solenoidal-flux-space}
\diver_yP=0
\qquad
\text{for every }P\in\mathcal P_{n_0}.
\end{equation}
Define
\begin{align}
\Psi_{0,n_0}
&:=
\inf_{P\in\mathcal P_{n_0}}
\norm{P-P^0}_{W^{1,\infty}(\Om\times Y)},
\label{eq:Psi0}\\
\Psi_{1,n_1}
&:=
\inf_{Q\in\mathcal Q_{n_1}}
\norm{Q-Q^0}_{W^{1,\infty}(\Om\times Y)}.
\label{eq:Psi1}
\end{align}
For $n=n_0+n_1$, define the flux class
\begin{equation}
\label{eq:two-scale-flux-class}
\mathcal P_n^\eps
:=
\left\{
P\!\left(x,\frac{x}{\eps}\right)
+
\eps
Q\!\left(x,\frac{x}{\eps}\right)
:
P\in\mathcal P_{n_0},
\ Q\in\mathcal Q_{n_1}
\right\}.
\end{equation}
Its divergence is
\begin{equation}
\label{eq:two-scale-flux-divergence}
\diver
\left[
P(\cdot,\cdot/\eps)
+
\eps Q(\cdot,\cdot/\eps)
\right]
=
\diver_xP(\cdot,\cdot/\eps)
+
\diver_yQ(\cdot,\cdot/\eps)
+
\eps\diver_xQ(\cdot,\cdot/\eps),
\end{equation}
so the $L^\infty$ envelopes of both the flux and its divergence are
uniform in $\eps$.

\begin{proposition}[Closure of the dual approximation term]
\label{prop:dual-closure}
Assume the parameter ball contains near-best approximants from
\eqref{eq:Psi0}--\eqref{eq:Psi1}. Then
\begin{equation}
\label{eq:dual-closure}
\mathfrak Q_n^\eps
\leq
C_{\mathrm{flux}}
\left(
\eps
+
\Psi_{0,n_0}^2
+
\Psi_{1,n_1}^2
\right),
\end{equation}
where $C_{\mathrm{flux}}$ is independent of $\eps$.
\end{proposition}

\begin{proof}
Choose $P_n\in\mathcal P_{n_0}$ and
$Q_n\in\mathcal Q_{n_1}$ near the infima in
\eqref{eq:Psi0}--\eqref{eq:Psi1}, and set
\[
p_n^\eps(x)
:=
P_n\!\left(x,\frac{x}{\eps}\right)
+
\eps
Q_n\!\left(x,\frac{x}{\eps}\right).
\]
The $L^2$ estimate gives
\[
\norm{p_n^\eps-\pi^\eps}_{L^2}
\leq
C\left(
\Psi_{0,n_0}+\Psi_{1,n_1}
\right).
\]
Using
\eqref{eq:solenoidal-flux-space} and
\eqref{eq:two-scale-flux-divergence},
\[
\norm{
\diver(p_n^\eps-\pi^\eps)
}_{L^2}
\leq
C\left(
\Psi_{0,n_0}+\Psi_{1,n_1}
\right).
\]
Combine these bounds with
\cref{prop:flux-Hdiv-corrector} and use
$(a+b+c)^2\leq3(a^2+b^2+c^2)$ in
\eqref{eq:pd-dual-approx}.
\end{proof}

The state approximation term also satisfies
\begin{equation}
\label{eq:pd-state-closure}
\mathfrak E_m^\eps
\leq
C_{\mathrm{state}}
\left(
\eps+\Phi_{0,m_0}^2+\Phi_{1,m_1}^2
\right)
\end{equation}
by the proof of \cref{thm:md-approximation-closure} and the upper
curvature bounds in \eqref{eq:pd-curvature}.

\begin{theorem}[Primal--dual bound under flux-corrector regularity]
\label{thm:pd-closed}
Let $(\beta_K,\gamma_K)$ be generated by projected gradient descent on
the empirical primal--dual gap using the state class
\eqref{eq:md-trial-class} and flux class
\eqref{eq:two-scale-flux-class}. Under
\cref{ass:structure,ass:corrector-regime,ass:pd-regularity,ass:flux-corrector}, with probability at least $1-\delta$,
\begin{align}
&\frac\alpha2
\norm{v_{\beta_K}-u^\eps}_{\V}^2
+
\frac\mu2
\norm{v_{\beta_K}-u^\eps}_{L^2(\Om)}^2
\notag\\
&\qquad\leq
C_{\mathrm{pd}}
\left(
\eps
+
\Phi_{0,m_0}^2
+
\Phi_{1,m_1}^2
+
\Psi_{0,n_0}^2
+
\Psi_{1,n_1}^2
\right)
\notag\\
&\qquad\quad+
2\Delta_{N,\delta}^{\mathrm{pd}}
+
\frac{
2L_{\mathrm{par}}^{\mathrm{pd}}(R^2+S^2)
}{K}.
\label{eq:pd-closed}
\end{align}
Every sampling, stability, and finite-iteration constant is independent
of $\eps$.
\end{theorem}

\begin{proof}
On the event in \cref{lem:pd-uniform-deviation}, compare
$(\beta_K,\gamma_K)$ with a population near-minimizer and insert the
empirical gap twice. Equation \eqref{eq:pd-pgd-rate} gives
\[
\mathfrak G_\eps(v_{\beta_K},p_{\gamma_K})
\leq
\mathfrak A_{m,n}^{\eps,\mathrm{pd}}
+
2\Delta_{N,\delta}^{\mathrm{pd}}
+
\frac{
2L_{\mathrm{par}}^{\mathrm{pd}}(R^2+S^2)
}{K}.
\]
Apply \cref{thm:pd-certificate,prop:pd-separated-approximation},
then use \eqref{eq:pd-state-closure} and
\cref{prop:dual-closure}.
\end{proof}

\begin{remark}
The condition $\diver_yP=0$ in \eqref{eq:solenoidal-flux-space} is
structural. Without it, differentiating
$P(x,x/\eps)$ would reintroduce an $\eps^{-1}\diver_yP$ term into the
dual-state feature and destroy the scale-uniform divergence envelope.
\end{remark}

\section{General strong-form statistical ill-conditioning}
\label{sec:strong-form}

We now show that the strong formulation exhibits the opposite scale
behavior. Specialize to
\begin{equation}
\label{eq:sf-general-pde}
-\diver A\!\left(\frac{x}{\eps},\grad u^\eps\right)
+
g(u^\eps)
=0
\quad\text{in }\Om,
\qquad
u^\eps=0
\quad\text{on }\partial\Om,
\end{equation}
where
\begin{equation}
\label{eq:sf-zero-normalization}
A(y,0)=0,
\qquad
g(0)=0.
\end{equation}
Then $u^\eps\equiv0$ is a solution.

For
$v\in W^{2,\infty}(\Om)\cap H_0^1(\Om)$, define
\begin{equation}
\label{eq:sf-residual-general}
\mathscr R_\eps[v]
:=
-\diver
A\!\left(
\frac{x}{\eps},\grad v
\right)
+
g(v).
\end{equation}
Set
\begin{equation}
\label{eq:sf-H-def}
H(y,\xi):=\diver_yA(y,\xi).
\end{equation}
The chain rule gives
\begin{equation}
\label{eq:sf-chain-rule}
\mathscr R_\eps[v]
=
-\frac1\eps
H\!\left(
\frac{x}{\eps},\grad v
\right)
-
D_\xi A\!\left(
\frac{x}{\eps},\grad v
\right):D^2v
+
g(v).
\end{equation}

\begin{assumption}[Strong-form nondegeneracy]
\label{ass:sf-general}
The flux belongs to
$C_{\mathrm{per}}^2(Y\times\R^d;\R^d)$. Fix
\[
v_\star
\in
W^{2,\infty}(\Om)\cap H_0^1(\Om),
\]
and assume that $D_\xi A$, $H$, $D_yH$, and $D_\xi H$ are bounded on
the compact gradient range of $v_\star$. Assume also that
$g(v_\star)\in L^\infty(\Om)$ and
\begin{equation}
\label{eq:sf-nondegeneracy}
\Theta_\star
:=
\frac1{\abs{\Om}}
\int_\Om\int_Y
\left|
H\bigl(
y,\grad v_\star(x)
\bigr)
\right|^2
\dd y\dd x
>0.
\end{equation}
\end{assumption}

\subsection{Periodic averaging}

\begin{lemma}[Two-scale averaging along a smooth field]
\label{lem:sf-averaging}
Let $K\subset\R^q$ be compact,
$b\in C_{\mathrm{per}}^1(Y\times K)$, and
$\zeta\in W^{1,\infty}(\Om;K)$. With
\[
\overline b(\zeta)
:=
\int_Yb(y,\zeta)\dd y,
\]
there is $C_{\mathrm{av}}>0$ such that
\begin{equation}
\label{eq:sf-averaging}
\left|
\int_\Om
b\!\left(
\frac{x}{\eps},\zeta(x)
\right)\dd x
-
\int_\Om
\overline b(\zeta(x))\dd x
\right|
\leq
C_{\mathrm{av}}\eps.
\end{equation}
\end{lemma}

\begin{proof}
For fixed $\zeta$, let $\varphi(\cdot,\zeta)$ be the mean-zero periodic
solution of
\[
-\Delta_y\varphi
=
b-\overline b,
\]
and set $\Psi=-\grad_y\varphi$. Periodic elliptic regularity, compactness
of $K$, and the assumed $C^1$ bounds give
\[
\norm{\Psi}_{L^\infty(Y\times K)}
+
\norm{D_\zeta\Psi}_{L^\infty(Y\times K)}
\leq C.
\]
Since $\diver_y\Psi=b-\overline b$,
\[
b\!\left(
\frac{x}{\eps},\zeta(x)
\right)
-
\overline b(\zeta(x))
=
\eps
\diver_x
\Psi\!\left(
\frac{x}{\eps},\zeta(x)
\right)
-
\eps
D_\zeta\Psi\!\left(
\frac{x}{\eps},\zeta(x)
\right):\grad\zeta(x).
\]
Integrate over $\Om$ and apply the divergence theorem.
\end{proof}

\subsection{Residual and squared-loss lower bounds}

Set
\begin{equation}
\label{eq:sf-d-eps}
d_\eps(x)
:=
\mathscr R_\eps[v_\star](x).
\end{equation}
Then
\begin{equation}
\label{eq:sf-decomposition}
d_\eps
=
-\frac1\eps
H\!\left(
\frac{x}{\eps},\grad v_\star
\right)
+
B_\eps,
\end{equation}
where
\[
B_\eps
:=
-
D_\xi A\!\left(
\frac{x}{\eps},\grad v_\star
\right):D^2v_\star
+
g(v_\star).
\]
Define
\[
M_H
:=
\esssup_{x,y}
\left|
H(y,\grad v_\star(x))
\right|,
\qquad
M_B
:=
\sup_{0<\eps\leq1}
\norm{B_\eps}_{L^\infty(\Om)},
\]
and
\begin{equation}
\label{eq:sf-U-mu}
U_\star:=(M_H+M_B)^2,
\qquad
\mu_\star:=\frac{\Theta_\star}{8}.
\end{equation}

\begin{lemma}[Oscillatory mass survives sampling]
\label{lem:sf-sampled-mass}
There exists $\eps_\star>0$ such that, for
$0<\eps\leq\eps_\star$,
\begin{equation}
\label{eq:sf-population-mass}
\frac1{\abs{\Om}}
\int_\Om
\eps^2d_\eps^2\dd x
\geq
\mu_\star.
\end{equation}
If $X_1,\ldots,X_N$ are independent uniform samples on $\Om$ and
\begin{equation}
\label{eq:sf-sample-threshold}
N
\geq
\frac{2U_\star^2}{\mu_\star^2}
\log\frac1\delta,
\end{equation}
then, with probability at least $1-\delta$,
\begin{equation}
\label{eq:sf-empirical-mass}
\frac1N
\sum_{i=1}^N
d_\eps(X_i)^2
\geq
\frac{\mu_\star}{2\eps^2}.
\end{equation}
\end{lemma}

\begin{proof}
Apply \cref{lem:sf-averaging} to
\[
b(y,\xi)=\abs{H(y,\xi)}^2,
\qquad
\zeta(x)=\grad v_\star(x).
\]
For sufficiently small $\eps$,
\[
\frac1{\abs{\Om}}
\int_\Om
\left|
H\!\left(
\frac{x}{\eps},\grad v_\star
\right)
\right|^2
\dd x
\geq
\frac{\Theta_\star}{2}.
\]
Using $\abs{a+b}^2\geq\abs a^2/2-\abs b^2$ in
\eqref{eq:sf-decomposition},
\[
\frac1{\abs{\Om}}
\int_\Om
\eps^2d_\eps^2\dd x
\geq
\frac{\Theta_\star}{4}
-
\eps^2M_B^2.
\]
Choose $\eps_\star$ so that the last term is at most
$\Theta_\star/8$.

Now set $Z_i=\eps^2d_\eps(X_i)^2$. Then
$0\leq Z_i\leq U_\star$ and $\E Z_i\geq\mu_\star$.
Hoeffding's inequality gives
\[
\Pp\left(
\frac1N\sum_{i=1}^NZ_i
<
\frac{\mu_\star}{2}
\right)
\leq
\exp\left(
-\frac{N\mu_\star^2}{2U_\star^2}
\right).
\]
Use \eqref{eq:sf-sample-threshold}.
\end{proof}

For a class $\mathcal F$ and sample $S=(X_i)_{i=1}^N$, define
\begin{equation}
\label{eq:rad-definition}
\Radh_N(\mathcal F;S)
:=
\E_\sigma
\left[
\sup_{h\in\mathcal F}
\frac1N
\sum_{i=1}^N
\sigma_i h(X_i)
\right].
\end{equation}
Because $\mathscr R_\eps[0]=0$, define
\begin{equation}
\label{eq:sf-two-point-classes}
\mathcal R_{\eps,\star}
:=
\{0,d_\eps\},
\qquad
\mathcal L_{\eps,\star}
:=
\{0,d_\eps^2\}.
\end{equation}

\begin{theorem}[General strong-form statistical blow-up]
\label{thm:strong-lower}
Under \cref{ass:sf-general}, if
$0<\eps\leq\eps_\star$ and
\eqref{eq:sf-sample-threshold} holds, then with probability at least
$1-\delta$,
\begin{align}
\Radh_N(\mathcal R_{\eps,\star};S)
&\geq
\frac{\sqrt{\mu_\star}}{4}
\frac1{\eps\sqrt N},
\label{eq:sf-residual-lower}\\
\Radh_N(\mathcal L_{\eps,\star};S)
&\geq
\frac{\mu_\star}{4\sqrt2}
\frac1{\eps^2\sqrt N}.
\label{eq:sf-squared-lower}
\end{align}
\end{theorem}

\begin{proof}
By symmetry of the Rademacher sum,
\begin{align*}
\Radh_N(\mathcal R_{\eps,\star};S)
&=
\frac1{2N}
\E_\sigma
\left|
\sum_{i=1}^N
\sigma_i d_\eps(X_i)
\right|
\\
&\geq
\frac1{2\sqrt2\,N}
\left(
\sum_{i=1}^N
d_\eps(X_i)^2
\right)^{1/2}.
\end{align*}
Use \eqref{eq:sf-empirical-mass}.

Similarly,
\[
\Radh_N(\mathcal L_{\eps,\star};S)
\geq
\frac1{2\sqrt2\,N}
\left(
\sum_{i=1}^N
d_\eps(X_i)^4
\right)^{1/2}.
\]
Cauchy--Schwarz gives
\[
\left(
\sum_i d_\eps(X_i)^4
\right)^{1/2}
\geq
\frac1{\sqrt N}
\sum_i d_\eps(X_i)^2.
\]
Use \eqref{eq:sf-empirical-mass} again.
\end{proof}

\begin{remark}[Scope with respect to adapted trial classes]
	\label{rem:adapted-classes}
	Theorem~\ref{thm:strong-lower} uses a fixed state
	\(v_\star\), independent of \(\eps\), and the two-point class generated by
	its residual. Consequently, any larger strong-form class containing
	\(0\) and \(v_\star\) inherits the same lower bound. The theorem does not
	apply directly to families \(v_\eps\) designed using the microscopic
	coefficient. In special structured settings, such as one-dimensional
	layered linear media, coefficient-adapted harmonic coordinates can cancel
	the leading \(\eps^{-1}\) contribution. Whether analogous cancellations
	exist for general nonlinear multidimensional fluxes, and whether they can
	be achieved without increasing approximation or model complexity, remains
	open.
\end{remark}

\begin{remark}[Formulation-level separation]
The lower bounds are properties of the sampled function classes and are independent of the optimizer. By contrast, the primal and primal--dual sampling estimates involve bounded state and flux envelopes but no derivative in the fast variable. Thus, under the additional regularity and nondegeneracy assumptions of this section, the strong obstruction and the scale-robust upper theory apply to the same variationally generated nonlinear periodic flux class.
\end{remark}

\section{Numerical validation}
\label{sec:numerics}

The numerical study tests the finite-sample signatures of the theory rather than attempting to replace its analytical assumptions.  We address six questions: whether the predicted powers of the microscopic scale $\eps$ are observed; whether the sample-size dependence is $N^{-1/2}$; whether the appropriately normalized complexities collapse across simultaneous changes of $(\eps,N)$; whether the exponents persist for genuinely nonlinear fluxes and in dimensions $d=1,2,3$; whether the oscillatory-mass event used in \cref{lem:sf-sampled-mass} concentrates at a scale-independent sample size; and whether the computable primal--dual gap remains above the true energy error while corrector enrichment improves state approximation.

\subsection{Models, estimators, and reproducibility}
\label{subsec:num-setup}

\paragraph{Baseline strong-form witness.}
For the scale and sample-size experiments we use $\Om=(0,1)$,
\begin{equation}
\label{eq:num-baseline-model}
A(y,\xi)=a(y)\xi,
\qquad
a(y)=2+\sin(2\pi y),
\qquad
g(s)=s+s^3,
\end{equation}
and
\begin{equation}
\label{eq:num-baseline-witness}
\psi(x)=x(1-x),
\qquad
v_c(x)=c\psi(x),
\qquad
\abs c\leq\kappa,
\qquad
\kappa=1.
\end{equation}
The strong residual and energy density are
\begin{align}
r_c^\eps(x)
&=
-\frac{c}{\eps}
a'\!\left(\frac{x}{\eps}\right)\psi'(x)
+2c\,a\!\left(\frac{x}{\eps}\right)
+c^3\psi(x)^3+c\psi(x),
\label{eq:num-residual}\\
e_c^\eps(x)
&=
\frac12a\!\left(\frac{x}{\eps}\right)c^2\psi'(x)^2
+\frac12c^2\psi(x)^2
+\frac14c^4\psi(x)^4.
\label{eq:num-energy-density}
\end{align}
Writing $d_\eps=r_\kappa^\eps$, we evaluate the symmetric residual class
$\{-d_\eps,d_\eps\}$, the squared-loss class $\{0,d_\eps^2\}$, and the
energy class $\mathcal E_\eps=\{e_c^\eps:\abs c\leq\kappa\}$.  For a spatial
sample $S=(X_i)_{i=1}^N$, their empirical complexities are estimated as
\begin{align}
\widehat{\mathfrak R}_{N}^{\rm res}
&=
\E_\sigma
\left|
\frac1N\sum_{i=1}^N\sigma_i d_\eps(X_i)
\right|,
\label{eq:num-res-estimator}\\
\widehat{\mathfrak R}_{N}^{\rm sq}
&=
\frac12\E_\sigma
\left|
\frac1N\sum_{i=1}^N\sigma_i d_\eps(X_i)^2
\right|.
\label{eq:num-sq-estimator}
\end{align}
For the energy class, $t=c^2\in[0,\kappa^2]$ gives
$e_c^\eps=tA_\eps+t^2B$.  For each Rademacher vector, the signed empirical
sum is therefore a scalar quadratic polynomial in $t$; its supremum is
computed exactly from the two endpoints and any stationary point contained
in $[0,\kappa^2]$.  Thus Monte Carlo error enters only through the
Rademacher expectation, not through the inner supremum.

\paragraph{Nonlinear multidimensional robustness model.}
To test that the strong-form exponents are not artifacts of linearity or of
one space dimension, we use $\Om=(0,1)^d$ and
\begin{align}
A_j(y,\xi)
&=
a(y_j)\xi_j
+\gamma b(y_j)\tanh(\xi_j),
\qquad
b(y)=1+\frac14\cos(2\pi y),
\label{eq:num-nonlinear-flux}\\
\phi_d(x)
&=
\prod_{j=1}^d x_j(1-x_j),
\qquad
v_c(x)=c\phi_d(x).
\label{eq:num-tensor-witness}
\end{align}
Because
$\partial_{\xi_j}A_j=a(y_j)+\gamma b(y_j)\operatorname{sech}^2(\xi_j)>0$,
this is a genuinely nonlinear uniformly monotone flux for every tested
$\gamma$.  We evaluate $d\in\{1,2,3\}$ and
$\gamma\in\{0,0.4,0.8,1.2\}$.

\paragraph{Primal--dual certificate problem.}
The certificate experiment solves
\begin{equation}
\label{eq:num-certificate-pde}
-\frac{\dd}{\dd x}
\left[
a\!\left(\frac{x}{\eps}\right)u^{\eps\prime}(x)
\right]
+(u^\eps)^3+u^\eps
=1,
\qquad
u^\eps(0)=u^\eps(1)=0.
\end{equation}
A conservative midpoint-flux finite-difference discretization with
$16\,384$ cells provides the reference solution.  The nonlinear algebraic
system is solved by damped Newton iteration with an Armijo line search,
residual tolerance $10^{-10}$, Newton-step tolerance $10^{-12}$, at most
$80$ iterations, and minimum line-search step $2^{-24}$.  Population
integrals are evaluated on $16\,384$ midpoint nodes.

The coefficient-blind state class is
\begin{equation}
\label{eq:num-blind-class}
v_\beta(x)=\sum_{j=1}^m\beta_j\sin(j\pi x).
\end{equation}
For the enriched class, let $a_{\rm hom}=(\int_0^1a^{-1})^{-1}$ and let the
mean-zero periodic corrector satisfy
$\chi'(y)=a_{\rm hom}/a(y)-1$.  With a $C^1$ cutoff $\eta_\eps$ that vanishes
within distance $\eps$ of the boundary and equals one beyond distance
$2\eps$, the enriched features are
\begin{equation}
\label{eq:num-enriched-class}
z_j^\eps(x)
=
\sin(j\pi x)
+\eps\eta_\eps(x)\chi(x/\eps)
\frac{\dd}{\dd x}\sin(j\pi x).
\end{equation}
The corrector is tabulated on $65\,536$ periodic grid cells.  The flux class
contains a constant and sine/cosine harmonics up to order
$\max\{m,4\}$, giving $1+2\max\{m,4\}$ flux coefficients.  Both the primal
energy and dual functional are minimized by L--BFGS--B with analytic
gradients, function tolerance $10^{-13}$, gradient tolerance $10^{-10}$,
and at most $4000$ iterations.  The dual conjugate of
$G(s)=s^2/2+s^4/4$ is evaluated from the real solution of $s+s^3=r$.

\paragraph{Randomization and uncertainty quantification.}
All experiments use one NumPy random generator initialized with the master
seed
\begin{equation}
\label{eq:num-seed}
\texttt{20260717}.
\end{equation}
Independent spatial samples are regenerated across replicates.  Reported
means use independent spatial-sample replicates, and their displayed
$95\%$ intervals are normal intervals based on the replicate standard
error.  Log--log slopes are fitted by ordinary least squares and accompanied
by percentile intervals from $3000$ joint bootstrap resamples of the spatial
replicates.  The full run used Python~3.13.9 and NumPy~2.3.5 under Anaconda on
Windows~11 and completed in $625.3$ seconds; the runtime is reported only for
reproducibility because the hardware was not recorded.

All parameters are collected in \cref{tab:num-settings}.  The symbol
$M_\sigma$ denotes the number of Rademacher sign vectors, $R_x$ the number of
independent spatial-sample replicates, and $R_b$ the number of bootstrap
resamples.

\begin{table}[!htbp]
\centering
\caption{Complete configuration of the full numerical run.  Ranges written
as $2^{-a{:}b}$ include every integer exponent from $a$ to $b$.}
\label{tab:num-settings}
\small
\renewcommand{\arraystretch}{1.13}
\begin{tabular}{@{}p{0.25\textwidth}p{0.28\textwidth}p{0.39\textwidth}@{}}
\toprule
Experiment & Quantity & Value \\
\midrule
Global & master seed; amplitude & $20260717$; $\kappa=1$ \\
$\eps$ scaling & $\eps$; $N$; $R_x$; $M_\sigma$; $R_b$
& $2^{-3{:}9}$; $2048$; $16$; $4096$; $3000$ \\
$N$ scaling & fixed $\eps$; $N$; $R_x$; $M_\sigma$; $R_b$
& $2^{-8}$; $128,256,512,1024,2048,4096$; $16$; $4096$; $3000$ \\
Normalized collapse & $\eps$; $N$; $R_x$; $M_\sigma$
& $2^{-3{:}8}$; $256,512,1024,2048,4096$; $10$; $2048$ \\
Dimension/nonlinearity & $d$; $\gamma$; $\eps$; $N$; $R_x$; $M_\sigma$
& $1,2,3$; $0,0.4,0.8,1.2$; $2^{-3{:}7}$; $2048$; $10$; $2048$ \\
Concentration & $\eps$; $N$; trials; population grid
& $2^{-3{:}8}$; $32,64,128,256,512,1024,2048$; $3000$; $262\,144$ \\
Primal--dual & $\eps$; $m$; forcing; reference cells; quadrature nodes
& $2^{-3{:}7}$; $4,8,12,16$; $1$; $16\,384$; $16\,384$ \\
Primal--dual & corrector cells; L--BFGS tolerances; maximum iterations
& $65\,536$; $(10^{-13},10^{-10})$; $4000$ \\
\bottomrule
\end{tabular}
\end{table}

\subsection{Microscopic-scale and sample-size laws}
\label{subsec:num-scaling}

\Cref{fig:num-eps-scaling} displays the empirical complexities against
$1/\eps$ at fixed $N=2048$.  The residual and squared-loss curves are nearly
straight on log--log axes, while the energy curve remains flat.  At fixed
$\eps=2^{-8}$, \cref{fig:num-N-scaling} shows the common $N^{-1/2}$ decay.
The fitted slopes, bootstrap intervals, and coefficients of determination
are reported in \cref{tab:num-scaling-slopes}.

The estimates are
$0.9969$ for the residual and $1.9937$ for the squared loss, with bootstrap
intervals that tightly bracket the theoretical values $1$ and $2$.  The
energy slope is $-0.0011$, and its interval contains zero.  Its low
$R^2=0.0475$ is expected: the dependent variable is essentially constant,
so the remaining replicate variation dominates the negligible fitted trend.
For sample size, all three estimates lie within $0.004$ of $-1/2$, and all
three fits have $R^2>0.9997$.

\begin{table}[!htbp]
\centering
\caption{Log--log scaling estimates from the full run.  CI denotes the
percentile $95\%$ interval from $3000$ bootstrap resamples of the spatial
replicates.}
\label{tab:num-scaling-slopes}
\small
\begin{tabular}{@{}llrrrr@{}}
\toprule
Varied quantity & Observable & fitted slope & bootstrap CI & $R^2$ & predicted \\
\midrule
$1/\eps$ & strong residual & $0.9969$ & $[0.9952,0.9986]$ & $0.999978$ & $1$ \\
$1/\eps$ & squared strong loss & $1.9937$ & $[1.9904,1.9973]$ & $0.999981$ & $2$ \\
$1/\eps$ & variational energy & $-0.0011$ & $[-0.0046,0.0025]$ & $0.047505$ & $0$ \\
$N$ & strong residual & $-0.4979$ & $[-0.5031,-0.4925]$ & $0.999748$ & $-1/2$ \\
$N$ & squared strong loss & $-0.4964$ & $[-0.5046,-0.4874]$ & $0.999778$ & $-1/2$ \\
$N$ & variational energy & $-0.4995$ & $[-0.5063,-0.4924]$ & $0.999761$ & $-1/2$ \\
\bottomrule
\end{tabular}
\end{table}

\begin{figure}[!htbp]
\centering
\begin{subfigure}[t]{0.32\textwidth}
\centering
\includegraphics[width=\linewidth]{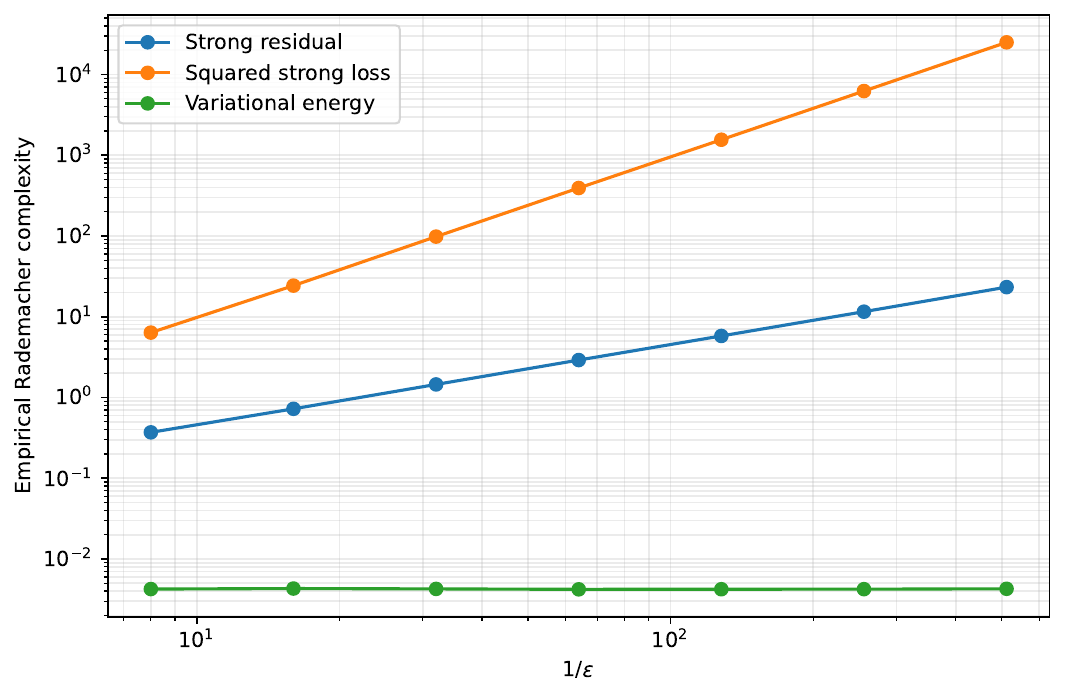}
\caption{Dependence on $1/\eps$ at $N=2048$.}
\label{fig:num-eps-scaling}
\end{subfigure}\hfill
\begin{subfigure}[t]{0.32\textwidth}
\centering
\includegraphics[width=\linewidth]{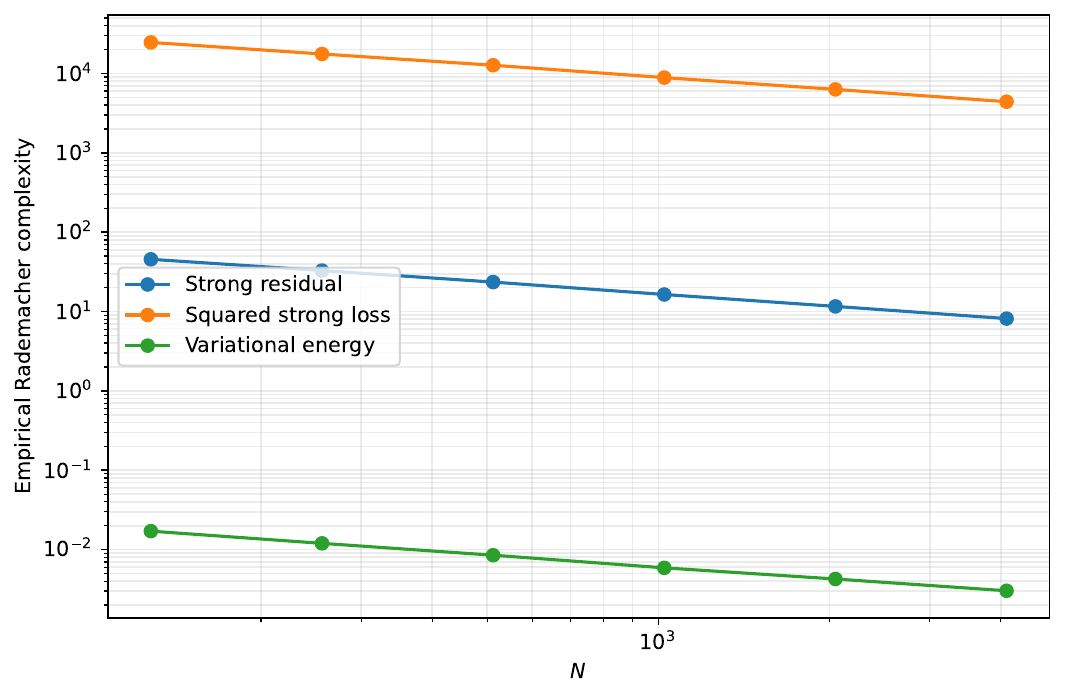}
\caption{Dependence on $N$ at $\eps=2^{-8}$.}
\label{fig:num-N-scaling}
\end{subfigure}\hfill
\begin{subfigure}[t]{0.32\textwidth}
\centering
\includegraphics[width=\linewidth]{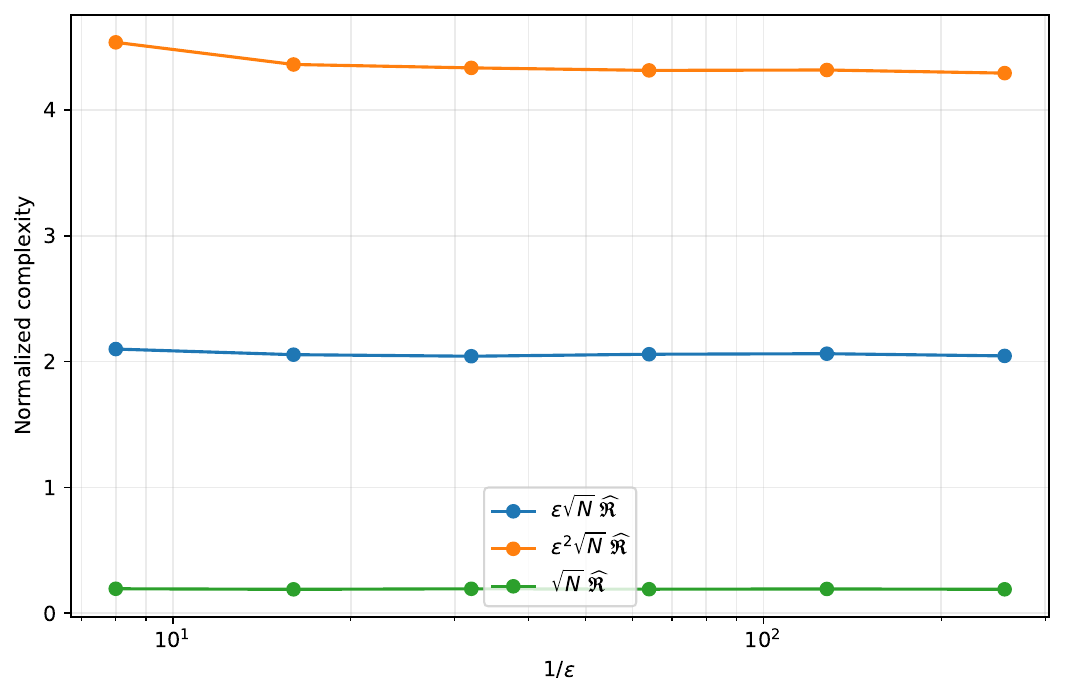}
\caption{Two-parameter normalized collapse.}
\label{fig:num-collapse}
\end{subfigure}
\caption{Finite-sample scaling of the strong residual, squared strong loss,
and variational energy.  Panels (a) and (b) verify the separate $\eps$ and
$N$ powers; panel (c) tests both variables simultaneously.}
\label{fig:num-scaling}
\end{figure}

\subsection{Normalized collapse and quantitative lower bounds}
\label{subsec:num-collapse}

A stronger test than two separate regressions is obtained by rescaling the
observables according to the theorem.  Across the $6\times5$ grid of
$(\eps,N)$ values, \cref{fig:num-collapse} shows that
\begin{equation}
\label{eq:num-collapse-quantities}
\eps\sqrt N\,\widehat{\mathfrak R}_{N}^{\rm res},
\qquad
\eps^2\sqrt N\,\widehat{\mathfrak R}_{N}^{\rm sq},
\qquad
\sqrt N\,\widehat{\mathfrak R}_{N}^{\rm en}
\end{equation}
are nearly constant.  Their coefficients of variation are only
$1.34\%$, $2.52\%$, and $1.46\%$, respectively, as summarized in
\cref{tab:num-collapse-bounds}.

The same table compares the empirical values with the explicit constants in
\cref{thm:strong-lower}.  The residual complexity is between $9.01$ and
$9.24$ times its proved lower bound, and the squared-loss complexity is
between $29.56$ and $30.96$ times its proved lower bound.  The purpose is not
to claim sharp constants, but to show that the ratio remains stable as
$\eps$ changes by a factor of $64$.

\begin{table}[!htbp]
\centering
\caption{Normalized-collapse statistics and ratios to the explicit theorem
lower bounds.  CV is computed over all $30$ tested $(\eps,N)$ pairs.}
\label{tab:num-collapse-bounds}
\small
\begin{tabular}{@{}lrrr@{}}
\toprule
Observable & normalized mean & CV & empirical/theorem range \\
\midrule
strong residual & $2.0605$ & $1.34\%$ & $[9.014,9.237]$ \\
squared strong loss & $4.3599$ & $2.52\%$ & $[29.557,30.958]$ \\
variational energy & $0.19160$ & $1.46\%$ & --- \\
\bottomrule
\end{tabular}
\end{table}

\subsection{Robustness to nonlinearity and dimension}
\label{subsec:num-dimension}

The nonlinear flux \eqref{eq:num-nonlinear-flux} probes the general setting
of \cref{thm:strong-lower}.  The fitted slopes in
\cref{fig:num-dimension} and \cref{tab:num-dim-nonlinearity} remain close to
$1$ and $2$ for every combination of $d$ and $\gamma$.  The largest absolute
deviations are $0.0368$ for the residual and $0.0839$ for the squared loss,
and the smallest coefficients of determination are $0.999410$ and
$0.999026$, respectively.  The modest downward drift at larger $d$ and
$\gamma$ is a finite-scale effect; it does not alter the observed separation
between first- and second-order powers.

\begin{table}[!htbp]
\centering
\caption{Fitted powers of $1/\eps$ for the genuinely nonlinear flux
\eqref{eq:num-nonlinear-flux}.}
\label{tab:num-dim-nonlinearity}
\small
\begin{tabular}{@{}rrrrrr@{}}
\toprule
$\gamma$ & $d$ & residual slope & residual $R^2$ & squared slope & squared $R^2$ \\
\midrule
$0.0$ & $1$ & $0.9931$ & $0.999953$ & $1.9884$ & $0.999965$ \\
$0.0$ & $2$ & $0.9899$ & $0.999927$ & $1.9851$ & $0.999969$ \\
$0.0$ & $3$ & $0.9859$ & $0.999917$ & $1.9596$ & $0.999822$ \\
$0.4$ & $1$ & $0.9889$ & $0.999981$ & $1.9845$ & $0.999975$ \\
$0.4$ & $2$ & $0.9852$ & $0.999929$ & $1.9593$ & $0.999891$ \\
$0.4$ & $3$ & $0.9738$ & $0.999801$ & $1.9420$ & $0.999817$ \\
$0.8$ & $1$ & $0.9854$ & $0.999908$ & $1.9734$ & $0.999945$ \\
$0.8$ & $2$ & $0.9813$ & $0.999702$ & $1.9506$ & $0.999722$ \\
$0.8$ & $3$ & $0.9706$ & $0.999605$ & $1.9251$ & $0.999241$ \\
$1.2$ & $1$ & $0.9880$ & $0.999961$ & $1.9793$ & $0.999950$ \\
$1.2$ & $2$ & $0.9700$ & $0.999738$ & $1.9342$ & $0.999560$ \\
$1.2$ & $3$ & $0.9632$ & $0.999410$ & $1.9161$ & $0.999026$ \\
\bottomrule
\end{tabular}
\end{table}

\subsection{Sampling concentration}
\label{subsec:num-concentration}

For each $\eps$, let
\begin{equation}
\label{eq:num-concentration-event}
\mathcal F_{\eps,N}
:=
\left\{
\frac1N\sum_{i=1}^N
\bigl(\eps d_\eps(X_i)\bigr)^2
<
\frac12\E\bigl[(\eps d_\eps(X))^2\bigr]
\right\}.
\end{equation}
The population expectation is approximated with $262\,144$ midpoint nodes,
and the failure probability is estimated from $3000$ independent samples
for each $(\eps,N)$.  As shown in \cref{fig:num-concentration-panel} and
\cref{tab:num-concentration}, the largest observed failure probability is
$1.033\%$ at $N=32$, falls below $3.34\times10^{-4}$ at $N=64$, and is zero
in all $18\,000$ trials performed at each $N\geq128$ across the six values of
$\eps$.  This directly supports the scale-independent concentration
mechanism in \cref{lem:sf-sampled-mass}; in particular, the experiment does
not indicate a requirement $N\gtrsim\eps^{-1}$.

\begin{table}[!htbp]
\centering
\caption{Empirical probability of the lower-tail event
\eqref{eq:num-concentration-event}, aggregated over the six tested values of
$\eps$.  A reported zero means that no failure occurred in the corresponding
$6\times3000$ trials.}
\label{tab:num-concentration}
\small
\begin{tabular}{@{}rrr@{}}
\toprule
$N$ & mean failure probability & maximum failure probability \\
\midrule
$32$ & $8.389\times10^{-3}$ & $1.033\times10^{-2}$ \\
$64$ & $2.222\times10^{-4}$ & $3.333\times10^{-4}$ \\
$128$ & $0$ & $0$ \\
$256$ & $0$ & $0$ \\
$512$ & $0$ & $0$ \\
$1024$ & $0$ & $0$ \\
$2048$ & $0$ & $0$ \\
\bottomrule
\end{tabular}
\end{table}

\begin{figure}[!htbp]
\centering
\begin{subfigure}[t]{0.49\textwidth}
\centering
\includegraphics[width=\linewidth]{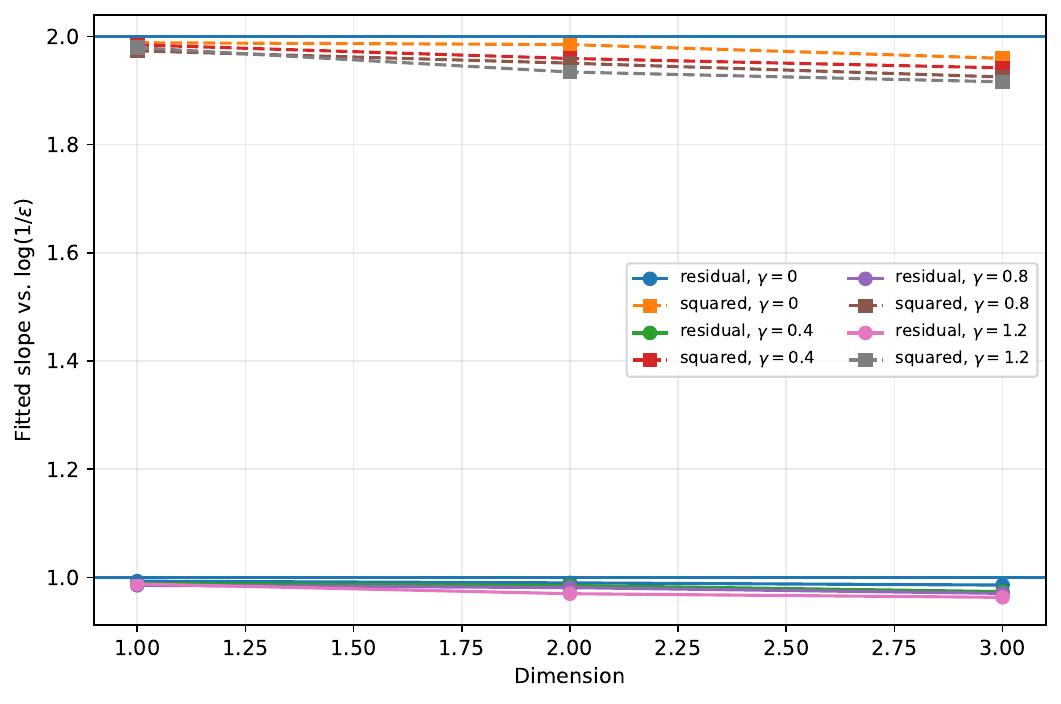}
\caption{Exponents versus dimension and nonlinear strength.}
\label{fig:num-dimension}
\end{subfigure}\hfill
\begin{subfigure}[t]{0.49\textwidth}
\centering
\includegraphics[width=\linewidth]{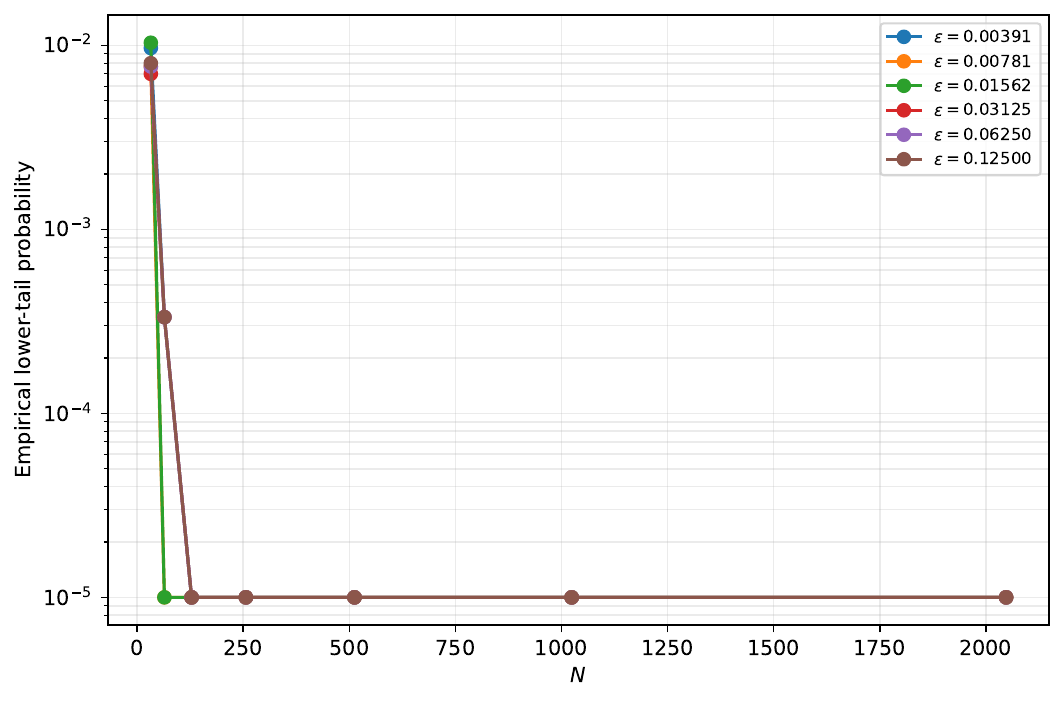}
\caption{Lower-tail concentration across $\eps$.}
\label{fig:num-concentration-panel}
\end{subfigure}
\caption{Robustness of the strong-form result.  Panel (a) tests nonlinear
fluxes in dimensions $1$--$3$; panel (b) tests the empirical-mass event used
in the proof.}
\label{fig:num-robustness}
\end{figure}

\subsection{Primal--dual reliability and corrector enrichment}
\label{subsec:num-primal-dual}

For every tested state and flux pair, the computed gap satisfies
\begin{equation}
\label{eq:num-certificate-inequality}
\mathfrak G_\eps(v_\beta,p_\gamma)
\geq
\J_\eps(v_\beta)-\J_\eps(u^\eps).
\end{equation}
The parity plot in \cref{fig:num-certificate-panel} contains all $40$ state
models.  The minimum certificate margin is
$6.43\times10^{-9}$, and the ratio of gap to true energy error lies in
$[1.000002,1.042244]$.  Thus the numerical dual approximation is sufficiently
accurate that the gap is nearly exact while remaining on the reliable side
of the diagonal.

\Cref{fig:num-state-panel} compares the state errors.  The coefficient-blind
class plateaus as $\eps$ decreases, whereas the corrector-enriched class
improves with the microscopic scale.  The width-$16$ values are reported in
\cref{tab:num-certificate}.  At $\eps=2^{-7}$, corrector enrichment reduces
the energy error by a factor $14.55$ and the squared $H^1$ error by a factor
$16.23$.  The increasing improvement factor as $\eps\downarrow0$ is the
expected signature of a class that resolves the microscopic gradient rather
than approximating only the macroscopic component.

\begin{table}[!htbp]
\centering
\caption{Primal and primal--dual results at state width $m=16$.  ``Factor''
is the coefficient-blind error divided by the corrector-enriched error.}
\label{tab:num-certificate}
\small
\resizebox{\textwidth}{!}{%
\begin{tabular}{@{}rrrrrrrr@{}}
\toprule
$\eps$ & blind energy & enriched energy & factor & blind $H^1$ sq. & enriched $H^1$ sq. & factor & gap/error \\
\midrule
$2^{-3}$ & $1.338\times10^{-3}$ & $9.52\times10^{-4}$ & $1.41$ & $1.826\times10^{-3}$ & $1.300\times10^{-3}$ & $1.41$ & $1.000025$ \\
$2^{-4}$ & $2.770\times10^{-3}$ & $1.183\times10^{-3}$ & $2.34$ & $4.076\times10^{-3}$ & $1.708\times10^{-3}$ & $2.39$ & $1.000144$ \\
$2^{-5}$ & $2.877\times10^{-3}$ & $6.92\times10^{-4}$ & $4.15$ & $4.240\times10^{-3}$ & $9.95\times10^{-4}$ & $4.26$ & $1.000101$ \\
$2^{-6}$ & $2.896\times10^{-3}$ & $3.76\times10^{-4}$ & $7.70$ & $4.274\times10^{-3}$ & $5.25\times10^{-4}$ & $8.15$ & $1.000100$ \\
$2^{-7}$ & $2.901\times10^{-3}$ & $1.99\times10^{-4}$ & $14.55$ & $4.282\times10^{-3}$ & $2.64\times10^{-4}$ & $16.23$ & $1.000166$ \\
\bottomrule
\end{tabular}}
\end{table}

\begin{figure}[!htbp]
\centering
\begin{subfigure}[t]{0.49\textwidth}
\centering
\includegraphics[width=\linewidth]{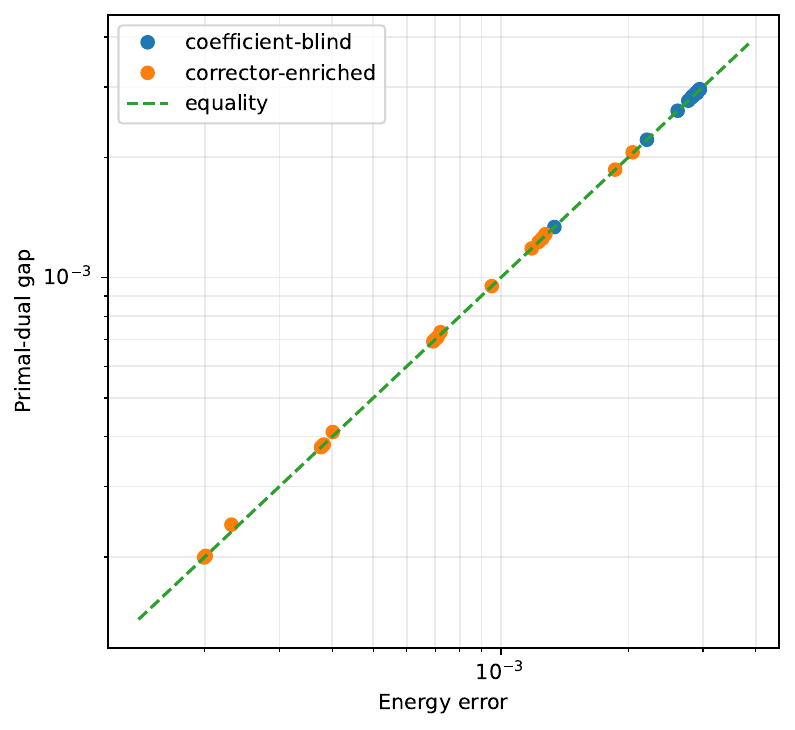}
\caption{Primal--dual gap against the true energy error.}
\label{fig:num-certificate-panel}
\end{subfigure}\hfill
\begin{subfigure}[t]{0.49\textwidth}
\centering
\includegraphics[width=\linewidth]{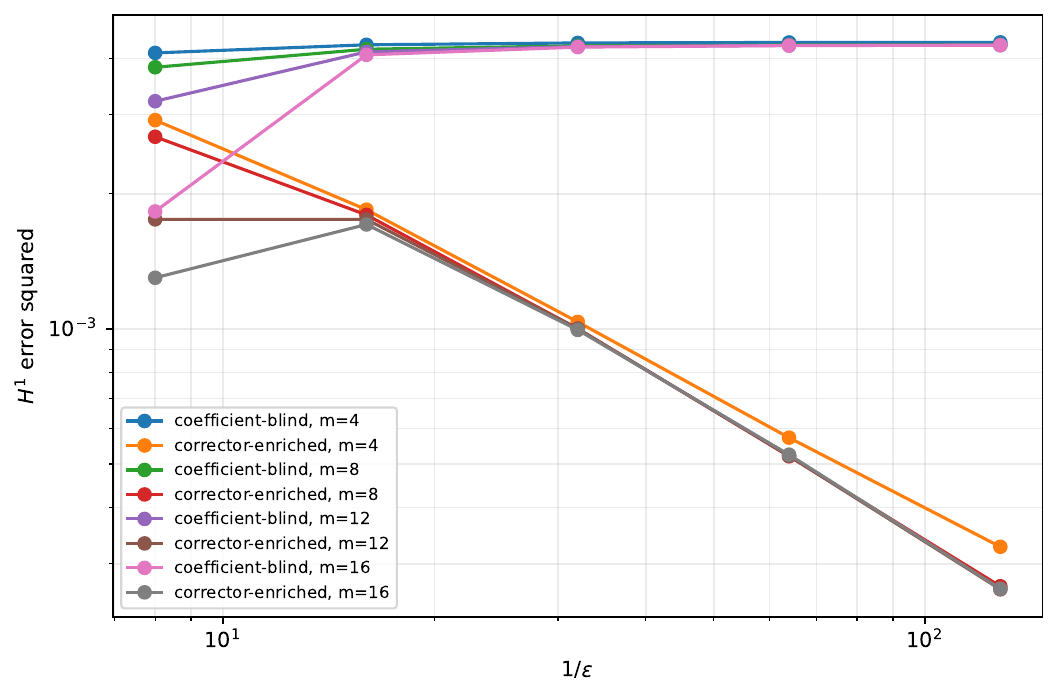}
\caption{Coefficient-blind and enriched state errors.}
\label{fig:num-state-panel}
\end{subfigure}
\caption{Numerical validation of the primal--dual theory.  Panel (a) verifies
the reliability inequality; panel (b) shows the effect of multiscale
corrector enrichment.}
\label{fig:num-primal-dual}
\end{figure}

\subsection{Summary and limitations of the numerical evidence}
\label{subsec:num-summary}

Figures~\ref{fig:num-scaling}, \ref{fig:num-robustness}, and \ref{fig:num-primal-dual}, together with Tables~\ref{tab:num-scaling-slopes}, \ref{tab:num-collapse-bounds}, \ref{tab:num-dim-nonlinearity}, \ref{tab:num-concentration}, and \ref{tab:num-certificate}, provide five independent forms of evidence, i.e., direct power-law regression, joint $(\eps,N)$ collapse, comparison with explicit lower-bound constants, robustness to dimension and flux nonlinearity, and numerical verification of the primal--dual certificate.

These computations support the finite-sample consequences of Theorems~\ref{thm:strong-lower} and~\ref{thm:pd-closed}. They do not prove the quantitative homogenization estimate or the additional flux-corrector regularity assumptions used in the analytical closure; those remain explicit hypotheses of the theory.

\section{Conclusion}
\label{sec:conclusion}

We established a non-asymptotic learning theory for uniformly monotone nonlinear multiscale elliptic equations that separates approximation, sampling, finite-iteration optimization, and formulation effects.

For the primal variational formulation, the state error satisfies an explicit approximation--sampling--optimization decomposition whose stability, empirical-quadrature, and projected-gradient constants are independent of the microscopic scale. Under a stated quantitative periodic corrector estimate, a two-scale state class closes the remaining approximation term in arbitrary spatial dimension.

The primal--dual formulation supplements the state approximation with a computable error certificate and a learned flux, without differentiating the microscopic coefficient. Under the additional flux-corrector regularity assumption, a skew-symmetric nonlinear flux corrector yields an \(H(\diver)\)-approximation of the exact flux, while microscopically divergence-free flux features preserve \(\eps\)-uniform sampling and optimization constants. This produces a conditional but closed state--flux end-to-end estimate in which all remaining scale dependence is confined to explicit state and flux approximation errors.

The strong formulation behaves differently. Whenever the microscopic flux divergence is nondegenerate along one admissible trial state, the empirical complexities of the strong residual and its squared loss grow at least as \(\eps^{-1}N^{-1/2}\) and \(\eps^{-2}N^{-1/2}\), respectively. These lower bounds are properties of the sampled function classes, are independent of the optimizer, and hold in every spatial dimension.

The numerical study confirms all observable scaling predictions. The fitted microscopic-scale exponents are \(0.9969\), \(1.9937\), and \(-0.0011\) for the residual, squared residual, and variational energy, respectively, while the corresponding sample-size exponents are \(-0.4979\), \(-0.4964\), and \(-0.4995\). After multiplication by the predicted powers of \(\eps\) and \(N\), the complexities exhibit a stable data collapse across the tested parameter grid. The same strong-form exponents persist for genuinely nonlinear monotone fluxes in dimensions \(d=1,2,3\), and the empirical lower-tail probabilities decrease with increasing sample size, consistently with the concentration argument.

The primal--dual computations provide complementary evidence. Every computed gap remains above the corresponding energy error, with certificate ratios close to one throughout the tested regime. Corrector enrichment also improves approximation as the scale is refined: at \(m=16\) and \(\eps=2^{-7}\), it reduces the energy error by a factor of \(14.55\) and the squared \(H^1\) error by a factor of \(16.23\) relative to the coefficient-blind class. These experiments do not replace the homogenization or flux-corrector assumptions; rather, they verify the finite-sample, dimensional, nonlinear, and certification signatures predicted by the analysis.

The strong-form result should therefore be interpreted as an obstruction for classes containing an \(\eps\)-independent nondegenerate witness, not as a universal impossibility theorem for every coefficient-adapted architecture. Characterizing the precise dichotomy between unavoidable statistical blow-up and successful microscopic adaptation is a natural direction for further work.

Therefore, the theory and computations identify the mechanism behind multiscale statistical ill-conditioning. Differentiating the fast coefficient amplifies the sampled strong loss, whereas variational and primal--dual formulations avoid that amplification. Corrector-enriched state classes and divergence-compatible flux classes then convert this formulation-level advantage into scale-robust non-asymptotic error bounds.

\bibliographystyle{unsrtnat}
\bibliography{refs1}
\end{document}